\documentclass{article}
\usepackage{helvet}  
\usepackage{hyperref}
\hypersetup{bookmarksnumbered} 
\usepackage[nottoc,notlot,notlof]{tocbibind} 
\usepackage{amssymb,amsmath,amsthm}
\usepackage{graphicx}
\usepackage{mathabx}  
\usepackage[subrefformat=parens,labelformat=parens]{subfig}
\graphicspath{{Pictures/}}   

\input xy
\xyoption{all}

\makeatletter

\let\c@table\c@figure
\makeatother

\newtheorem{theorem}{Theorem}[section]
\newtheorem{proposition}[theorem]{Proposition}
\newtheorem{lemma}[theorem]{Lemma}
\newtheorem{corollary}[theorem]{Corollary}

\theoremstyle{definition}
\newtheorem{definition}[theorem]{Definition}
\newtheorem{example}[theorem]{Example}

\newtheorem{convention}[theorem]{Convention}

\theoremstyle{remark}
\newtheorem{remark}[theorem]{Remark}

\newtheorem{notes}[theorem]{Notes}

\newcommand{\R}{\mathcal{R}}  

\newcommand{\rset}[1]{\mathcal{D}#1}

\newcommand{\newword}[1]{\textbf{#1}}

\newcommand{\edge}{\varepsilon}
\newcommand{\edgetwo}{\zeta}

\DeclareMathOperator{\CAT}{CAT}
\newcommand{\G}{\mathcal{G}}

\newcommand{\plaingraphics}[1]{\includegraphics{#1}}  
\newcommand{\vgraphics}[1]{\raisebox{-0.5\height}{\plaingraphics{#1}}}  
\newcommand{\longvarrow}{\raisebox{-0.5ex}{$\longrightarrow$}}  

\renewcommand{\exp}{\smalltriangleleft}

\newcommand{\cube}{\mathrm{cube}}
\newcommand{\Rloop}{R_{\mathrm{loop}}}

\title{Rearrangement Groups of Fractals}
\author{James Belk and Bradley Forrest}
\date{}  

\begin{document}

\maketitle

\pdfbookmark[1]{Introduction}{intro}
\begin{abstract}
\noindent We construct rearrangement groups for edge replacement systems, an infinite class of groups that generalize Richard Thompson's groups $F$, $T$, and~$V$.  Rearrangement groups act by piecewise-defined homeomorphisms on many self-similar topological spaces, among them the Vicsek fractal and many Julia sets.  We show that every rearrangement group acts properly on a locally finite $\mathrm{CAT}(0)$ cubical complex, and we use this action to prove that certain rearrangement groups are of type $F_{\infty}$.
\end{abstract}


\section*{Introduction}

In this paper we construct rearrangement groups, a class of groups that act by homeomorphisms on a large family of self-similar topological spaces.  This class includes Richard Thompson's groups $F$, $T$, and $V$, and many of the groups we construct have Thompson-like properties.  For example, we prove that every rearrangement group acts properly by isometries on a $\mathrm{CAT}(0)$ cubical complex, generalizing the complexes for $F$, $T$, and $V$ defined by Farley~\cite{Farley1,Farley2}.  By analyzing the geometry of these complexes, we show that certain rearrangement groups have type~$F_\infty$, generalizing results of Brown and Geoghegan~\cite{BroGeo,Bro}.

The spaces that these groups act on arise as limits of finite graphs.  Starting with a base graph~$G_0$, we repeatedly apply a certain edge replacement rule $e\to R$ to obtain a sequence $\{G_n\}$ of finite graphs.  This sequence converges to a limit space~$X$, which is usually a fractal space with a graph-like structure.  Figure~\ref{fig:LimitSpaces} shows two well-known fractals that can be obtained (up to homeomorphism) in this fashion: the basilica Julia set (the Julia set for~$z^2-1$) and the Vicsek fractal.
\begin{figure}
\centering
\subfloat[\label{subfig:Basilica}]{\vgraphics{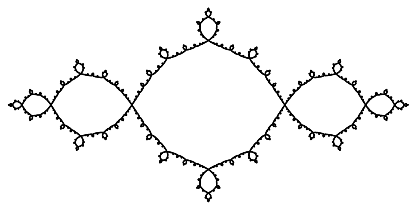}}
\hfill
\subfloat[\label{subfig:Vicsek}]{\vgraphics{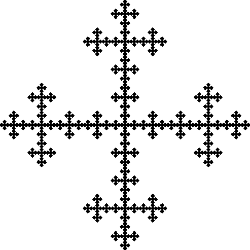}}
\caption{(a) The basilica Julia set. (b) The Vicsek fractal.}
\label{fig:LimitSpaces}
\end{figure}

By the nature of this construction, each limit space comes equipped with certain special subsets called cells, which correspond to edges of graphs in the sequence.  Each cell is topologically ``self-similar'' in the sense that it canonically homeomorphic with certain proper subspaces of itself.  A homeomorphism of the limit space is called a rearrangement if it preserves this structure, i.e.~if it locally maps cells to cells in the canonical way.  For example, Figure~\ref{fig:VicsekRearrangement} shows a rearrangement of the Vicsek fractal.
\begin{figure}
\centering
\vgraphics{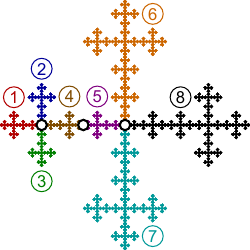}\qquad\longvarrow\qquad\vgraphics{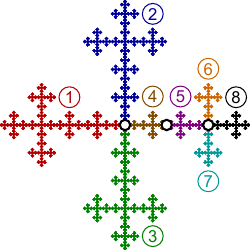}
\caption{A rearrangement of the Vicsek fractal.  Each numbered cell on the left maps to the corresponding cell on the right via a canonical homeomorphism.}
\label{fig:VicsekRearrangement}
\end{figure}

The set of all rearrangements of a limit space forms a group under composition, called the rearrangement group.  For a self-similar fractal in the plane such as the Vicsek fractal, the rearrangement group acts by piecewise-similar homeomorphisms, i.e.~homeomorphisms constructed by pasting together finitely many Euclidean similarity transformations.  For a limit space homeomorphic to a Julia set, the rearrangement group typically acts on the Julia set by piecewise-conformal homeomorphisms.

Thompson's groups $F$, $T$, and $V$ are special cases of rearrangement groups, corresponding to certain realizations of the closed interval, the circle, and the Cantor set as limit spaces.  In the case of Thompson's group~$F$, the cells of the limit space are precisely the standard dyadic intervals~$\bigl[i/2^j,(i+1)/2^j\bigr]$, with the canonical homeomorphism between two cells being the orientation-preserving linear map.  Thus $F$ is the group of rearrangements of the unit interval that preserves the self-similar structure defined by the standard dyadic intervals.

In~\cite{BeFo}, the authors described a group $T_B$ of homeomorphisms of the basilica Julia set, which were defined using piecewise-linear functions on the B\"{o}ttcher coordinates.  This group is also a special case of a rearrangement group, where the basilica is realized as a limit space of a sequence of graphs in an appropriate way.  We conjectured in~\cite{BeFo} that $T_B$ is not finitely presented, and this was recently proven by S.~Witzel and M.~Zaremsky using the $\mathrm{CAT}(0)$ complex we describe here~\cite{WiZa}.

The idea of constructing fractals as limits of sequences of finite graphs is not new.  For example, Laplacians on fractals are often constructed by realizing the fractal as the limit of a sequence of metric graphs~\cite{Ki,St}.  Certain Julia sets also arise as limits of sequences of finite Schreier graphs in Bartholdi and Nekrashevych's construction of iterated monodromy groups~\cite{Ne}.  However, with the exception of~\cite{BeFo}, it seems that groups of piecewise-similar homeomorphisms of such spaces have not previously been considered.

This paper is organized as follows.  Section~\ref{sec:rearrangement} introduces our terminology and context, defines rearrangement groups and limit spaces, and develops an analogue of tree pair diagrams for rearrangements. We begin to explore the resulting groups in Section~\ref{sec:RearrangementGroups}.  After discussing several examples, including two infinite families generalizing the basilica and Vicsek groups, we discuss the relationship between rearrangement groups and generalized Thompson groups, and we provide a classification of all finite subgroups of a rearrangement group.  We also introduce colored replacement systems, a generalization of edge replacement systems that allows us to construct rearrangement groups for a broader class of fractals. Section~\ref{sec:complexes} uses the techniques of Farley from \cite{Farley1} and \cite{Farley2} to construct a locally finite $\mathrm{CAT}(0)$ cubical complexes on which rearrangement groups act properly.  Together with Brown's criterion, and Bestvina and Brady's discrete Morse Theory, this action can be used to show that many rearrangement groups are of type $F_{\infty}$.  We produce a condition on the replacement system sufficient to show that the rearrangement group is of type $F_{\infty}$ in Theorem~\ref{thm:finfty} of Section~\ref{sec:finiteness}, and apply this theorem to show that all of the groups corresponding to the Vicsek family of fractals are of type~$F_{\infty}$.

\section{Limit Spaces and Rearrangements}
\label{sec:rearrangement}

In this section, we introduce limit spaces obtained from edge replacement rules and the corresponding rearrangement groups.  These groups are built by applying replacement rules to edges in graphs.  We set our context, defining replacement rules in Subsection~\ref{subsec:replacement}.  Repeated application of a replacement rule gives rise to a limit topological space, which we discuss in Subsection~\ref{subsec:limitspace}.  Subsection~\ref{subsec:cellrearrange} introduces rearrangements, which are a particular type of homeomorphism on these limit spaces that permute certain subsets, cells, of the limit space.  Graph pair diagrams, our primary graphical representation of elements of rearrangement groups, are developed in Subsection~\ref{subsec:graphpair}. Subsection~\ref{subsec:top} studies the topology of the limit space, and provides the technical details underpinning the theoretical development of this section.  Finally, Subsection~\ref{subsec:metrics} defines a natural family of metrics on the limit space. With respect to these metrics, rearrangements act as piecewise-similar homeomorphisms.

\subsection{Replacement Rules}
\label{subsec:replacement}

Throughout this paper, the word \textbf{graph} will refer to finite directed multigraphs, where both loops and multiple edges are allowed.  All isomorphisms between graphs are assumed to preserve the directions of the edges.

\begin{definition}An \newword{(edge) replacement rule} is a pair of the form $e\to R$, where
\begin{enumerate}
\item $e$ is a single (non-loop) directed edge, and
\item $R$ is a finite directed graph with distinguished vertices $v$ and $w$, which we require to be distinct.
\end{enumerate}
\end{definition}

The graph $R$ is called the \newword{replacement graph}, and the vertices $v$ and $w$ are referred to as the \newword{initial vertex} and \newword{terminal vertex} of~$R$.  Together these are the \newword{boundary vertices} of~$R$, and the remaining vertices of $R$ (if any) are \newword{interior vertices}.

Given a directed graph $G$ and a replacement rule $e\to R$, we can \newword{replace} (or \newword{expand}) any edge $\edge$ of $G$ by removing it and pasting in a copy of~$R$, attaching the initial and terminal vertices of $R$ respectively to the initial and terminal vertices of~$\edge$.  The resulting graph $G \exp \edge$ is called a \newword{simple expansion} of~$G$.

When discussing edge replacement, we adopt the convention that the vertices and edges of $G$ and $R$ can be treated as symbols from a finite alphabet, meaning that we are free to concatenate these symbols into sequences.  In particular, we will use the following notation for the new edges and vertices of~$G \exp \edge$:
\begin{enumerate}
\item Each new edge of $G \exp \edge$ has the form $\edge\edgetwo$, where $\edge$ is the edge of $G$ that was replaced, and $\edgetwo$ is any edge of~$R$.
\item Similarly, each new vertex of $G \exp \edge$ has the form $\edge\nu$, where $\nu$ is any interior vertex of~$R$.
\end{enumerate}

\begin{example}\label{ex:SimpleExpansion}Consider the replacement rule shown in Figure~\subref*{subfig:BasilicaReplacementRule}.  Note that the three edges of the replacement graph $R$ correspond to the symbols $\textsf{0}$, $\textsf{1}$, and~$\textsf{2}$, while the vertex corresponds to the symbol~$\textsf{v}$.
\begin{figure}
\centering
\subfloat[\label{subfig:BasilicaReplacementRule}]{\qquad\vgraphics{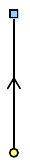} \quad\longvarrow\quad \vgraphics{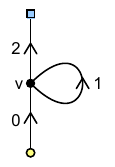}}
\qquad\qquad
\subfloat[\label{subfig:BasilicaBaseGraph}]{\vgraphics{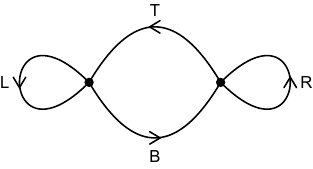}}
\caption{The basilica replacement system. Here and in the future, the initial and terminal vertices will be indicated by a yellow dot and a blue square, respectively.}
\label{fig:BasilicaReplacement}
\end{figure}

\begin{figure}
\centering
\subfloat[\label{subfig:SimpleExpansion}]{\vgraphics{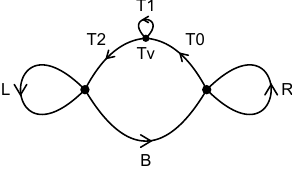}}
\hfill
\subfloat[\label{subfig:Expansion}]{\vgraphics{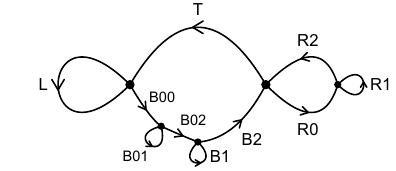}}
\caption{(a) A simple expansion of the graph in Figure~\protect\subref*{subfig:BasilicaBaseGraph}. (b) An expansion of that same graph.}
\label{fig:Expansions}
\end{figure}
Figure~\subref*{subfig:BasilicaBaseGraph} shows a directed graph $G$ with edges $\mathsf{L}$, $\mathsf{R}$, $\mathsf{T}$, $\mathsf{B}$, and Figure~\subref*{subfig:SimpleExpansion} shows the simple expansion $G \exp \mathsf{T}$.  Note that the edge $\textsf{T}$ of $G$ was replaced by a new subgraph with edges $\textsf{T0}$, $\textsf{T1}$, and $\textsf{T2}$ and a new vertex~$\textsf{Tv}$.
\end{example}

It is possible to iterate the process of simple expansion.  In general, an \newword{expansion} of a graph $G$ is any graph $E$ obtained from $G$ through a sequence of simple expansions.  Each new edge or vertex of $E$ can be described as a sequence.  In particular:
\begin{enumerate}
\item Each edge of $E$ finite sequence $\edge_0 \edge_1 \cdots \edge_n$ ($n\geq 0$), where $\edge_0$ is an edge of~$G$ and each $\edge_i$ for $1\leq i\leq n$ is an edge of $R$.
\item Each vertex of $E$ is either a vertex of~$G$, or is a finite sequence $\edge_0\edge_1\cdots\edge_n\nu$, where $\edge_0$ is an edge in~$G$, each $\edge_i$ for $1\leq i\leq n$ is an edge of~$R$, and $\nu$ is an interior vertex of~$R$.
\end{enumerate}

\begin{example}\label{ex:Expansion} Figure~\subref*{subfig:Expansion} shows the expansion $G \exp \mathsf{R} \exp \mathsf{B} \exp \mathsf{B0}$ of the graph $G$ from Figure~\subref*{subfig:BasilicaBaseGraph}, using the replacement rule shown in Figure~\subref*{subfig:BasilicaReplacementRule}.  Note that order in which the edges are replaced is irrelevant, in the sense that $G \exp \mathsf{R} \exp \mathsf{B} \exp \mathsf{B0}$, $G \exp \mathsf{B} \exp \mathsf{R} \exp \mathsf{B0}$, and $G \exp \mathsf{B} \exp \mathsf{B0} \exp \mathsf{R}$ all give the same expansion, though of course $\mathsf{B}$ must be expanded before~$\mathsf{B0}$.
\end{example}

\subsection{The Limit Space}
\label{subsec:limitspace}

\begin{definition}An \newword{(edge) replacement system} $\R$ is a pair $(G_0,e\to R)$, where $G_0$ is a finite, directed graph called the \newword{base graph}, and $e\to R$ is a replacement rule.
\end{definition}

Given a replacement system $(G_0,e\to R)$, the \newword{full expansion} of $G_0$ is the graph $G_1$ obtained by replacing every edge of~$G_0$.  Iterating this process, we obtain the \newword{full expansion sequence} $\{G_n\}_{n=0}^\infty$, where each $G_n$ is the full expansion of $G_{n-1}$.  Note that the edges of $G_n$ are precisely the elements of $E(G_0) \times E(R)^n$, where $E(G)$ denotes the set of edges of a graph~$G$.

\begin{figure}
\centering
$\underset{\textstyle G_1\rule{0pt}{12pt}}{\vgraphics{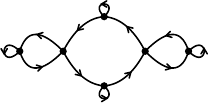}}
\quad
\underset{\textstyle G_2\rule{0pt}{12pt}}{\vgraphics{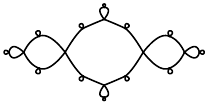}}
\quad
\underset{\textstyle G_3\rule{0pt}{12pt}}{\vgraphics{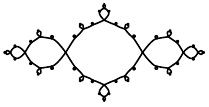}}$
\caption{Three graphs from the full expansion sequence for the basilica replacement system.  To improve readability, we have removed the arrowheads and vertex dots from the pictures of $G_2$ and~$G_3$.}
\label{fig:BasilicaFullExpansions}
\end{figure}
\begin{example}[The Basilica]\label{ex:BasilicaReplacement} Consider again the replacement system shown in Figure~\ref{fig:BasilicaReplacement}. The first few stages of the full expansion sequence for this replacement system are shown in Figure~\ref{fig:BasilicaFullExpansions}.

These graphs can be viewed as successive approximations to a fractal set known as the \newword{basilica}, which is shown in~Figure~\subref*{subfig:Basilica}.  The basilica is the Julia set for the function $f(z) = z^2 - 1$, i.e.~the boundary of the basin of infinity for this function.
\end{example}

The graphs in the full expansion sequence $\{G_n\}$ for a replacement system~$\R$ typically converge to a compact Hausdorff space~$X$, which we refer to as the ``limit space'' for~$\R$.  We shall define this limit space precisely as the quotient of a certain space of infinite sequences under an appropriate equivalence relation.

\begin{definition}The \newword{symbol space} $\Omega$ for a replacement system $(G_0,e\to R)$ is the space of all infinite sequences
\[
\edge_0\,\edge_1\,\edge_2\,\cdots
\]
where $\edge_0$ is an edge from~$G_0$, and each $\edge_i$ for $i\geq 1$ is an edge from~$R$.
\end{definition}

That is, the symbol space $\Omega$ is the infinite product $E(G_0) \times E(R)^\infty$, where $E(G_0)$ denotes the set of edges of $G_0$, and $E(R)$ denotes the set of edges of~$R$.  We endow $\Omega$ with the product topology.  Being an infinite product of finite sets, $\Omega$ is homeomorphic to the Cantor set (assuming $R$ has at least two edges).

\begin{definition}\label{def:LimitSpace}Let $\R = (G_0,e\to R)$ be a replacement system with full expansion sequence $\{G_n\}$ and symbol space~$\Omega$.  The \newword{gluing relation} on $\Omega$ is the equivalence relation $\sim$ defined as follows: two sequences
\[
\edge_0\,\edge_1\,\edge_2\,\cdots \qquad\text{and}\qquad \edge_0'\,\edge_1'\,\edge_2'\,\cdots
\]
are equivalent if for all $n$ the edges of $G_n$ with addresses
\[
\edge_0\,\edge_1\,\cdots\,\edge_n \qquad\text{and}\qquad \edge_0'\,\edge_1'\,\cdots\,\edge_n'
\]
share at least one vertex.  The \newword{limit space} $X$ for $\R$ is the quotient~$\Omega/{\sim}$.\end{definition}

Unfortunately, the gluing relation $\sim$ as defined above is not always an equivalence relation.  We will now state certain technical assumptions that need to be placed on a replacement system.

\begin{definition}
\label{def:expanding}A replacement system $\R = (G_0,e\to R)$ is \newword{expanding} if the following conditions are satisfied:
\begin{enumerate}
\item Neither $G_0$ nor $R$ has any isolated vertices.
\item The initial and terminal vertices of $R$ are not connected by an edge.
\item $R$ has at least three vertices and two edges.
\end{enumerate}
\end{definition}

Subsection \ref{subsec:top} discusses the point-set topology of the gluing relation and the limit space.  The following proposition is proven in Proposition~\ref{prop:GluingRelation} and Theorem~\ref{thm:hausdorff}.

\begin{proposition}If\/ $\R$ is an expanding replacement system, then the gluing relation~$\sim$ is an equivalence relation, and the limit space $X=\Omega/{\sim}$ is compact and metrizable.\hfill\qedsymbol
\end{proposition}

\begin{convention} From this point forward, all replacement systems are assumed to be expanding.
\end{convention}

\begin{example}[Gluing Relation for the Basilica] For the basilica replacement system given in Example~\ref{ex:BasilicaReplacement}, the symbol space is the infinite product
\[
\Omega = \{\mathsf{T},\mathsf{B},\mathsf{L},\mathsf{R}\} \times \{\mathsf{0},\mathsf{1},\mathsf{2}\}^\infty.
\]
It is not hard to work out the gluing relation on $\Omega$.  In particular:
\begin{enumerate}
\item Let $v$ denote the left vertex of the base graph~$G_0$.  Then any edge of the form $\mathsf{L0}^n$, $\mathsf{L2}^n$, $\mathsf{B0}^n$, or $\mathsf{T2}^n$ in $G_n$ is incident on~$v$.  It follows that the four points $\mathsf{L}\overline{\mathsf{0}}$, $\mathsf{L}\overline{\mathsf{2}}$, $\mathsf{B}\overline{\mathsf{0}}$, and $\mathsf{T}\overline{\mathsf{2}}$ in $\Omega$ are all equivalent under the gluing relation, where overline denotes repetition.
\item Similarly, the points $\mathsf{R}\overline{\mathsf{0}}$, $\mathsf{R}\overline{\mathsf{2}}$, $\mathsf{T}\overline{\mathsf{0}}$, and $\mathsf{B}\overline{\mathsf{2}}$ in $\Omega$ are all equivalent, corresponding to the right vertex of~$G_0$.
\item More generally, if $\edge_0\cdots\edge_n\textsf{v}$ is any vertex of~$G_{n+1}$, where $\textsf{v}$ denotes the interior vertex of~$R$, then
 \[
 \edge_0\cdots\edge_n \mathsf{0}\overline{\mathsf{2}} \;\;\sim\;\; \edge_0\cdots\edge_n \mathsf{1}\overline{\mathsf{0}} \;\;\sim\;\; \edge_0\cdots\edge_n \mathsf{1}\overline{\mathsf{2}} \;\;\sim\;\; \edge_0\cdots\edge_n \mathsf{2}\overline{\mathsf{0}}.
 \]
\end{enumerate}
All of the nontrivial equivalences under~$\sim$ are of one of the three forms listed above.  In particular, every point of $\Omega$ that does not end in an infinite sequence of $\mathsf{0}$'s or $\mathsf{2}$'s is a one-point equivalence class.  The quotient $X = \Omega/{\sim}$ is homeomorphic to the basilica Julia set shown in Figure~\subref*{subfig:Basilica}.  This can be proven using the standard description of the basilica as a quotient of the circle via the Thurston invariant lamination (see~\cite{Th}).
\end{example}

It is true in general that the nontrivial equivalence classes under the gluing relation correspond to vertices.  First, observe that the vertex sets for the graphs in the full expansion sequence $\{G_n\}$ form a nested chain
\[
V(G_0) \;\subset\; V(G_1) \;\subset\; V(G_2) \;\subset\; \cdots.
\]
We shall refer to elements of the union $\bigcup_{n=0}^\infty V(G_n)$ as \newword{gluing vertices} for the replacement system~$\R$. A point $\edge_0\edge_1\edge_2 \cdots \in \Omega$ \newword{represents} a gluing vertex $v$ if the edge $\edge_0\cdots\edge_n$ in $G_n$ is incident on $v$ for all sufficiently large~$n$.

We will prove in Proposition~\ref{prop:GluingRelation} that two distinct points in $\Omega$ are identified under the gluing relation if and only if they represent the same gluing vertex.  Moreover, the function that maps each gluing vertex to the corresponding point in $X$ is an injection.  From now on, we will identify each gluing vertex with its image in~$X$.  Thus $\edge_0 \edge_1 \edge_2\cdots$ represents a gluing vertex $v$ if and only if $\edge_0 \edge_1 \edge_2\cdots$ maps to $v$ under the quotient map $\Omega \to X$.

We now introduce our second main example of a replacement system and the corresponding limit space.

\begin{figure}
\centering
\vgraphics{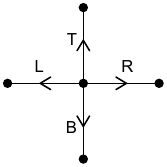}\qquad\qquad\qquad\vgraphics{EdgeReplacement1Tall} \quad\longvarrow\quad \vgraphics{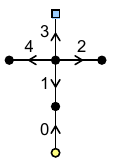}
\caption{The Vicsek replacement system.}
\label{fig:VicsekReplacement}
\end{figure}
\begin{figure}
\centering
$\underset{\textstyle G_1\rule{0pt}{12pt}}{\vgraphics{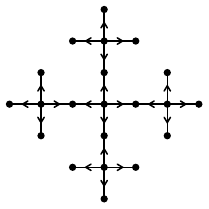}}
\qquad
\underset{\textstyle G_2\rule{0pt}{12pt}}{\vgraphics{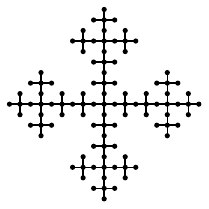}}
\qquad
\underset{\textstyle G_3\rule{0pt}{12pt}}{\vgraphics{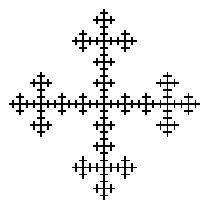}}$
\caption{Three graphs from the full expansion sequence for the Vicsek replacement system. To improve readability, we have removed the arrowheads from $G_2$ and~$G_3$, and we have removed the vertex dots from~$G_3$.}
\label{fig:VicsekESequ}
\end{figure}
\begin{example}[The Vicsek Fractal]
\label{ex:Vicsek}Consider the replacement system shown in Figure~\ref{fig:VicsekReplacement}.  The first few graphs in the full expansion sequence for this replacement system are shown in Figure~\ref{fig:VicsekESequ}.

%
The symbol space for this fractal is $\Omega = \{\mathsf{T},\mathsf{L},\mathsf{R},\mathsf{B}\}\times\{\mathsf{0},\mathsf{1},\mathsf{2},\mathsf{3},\mathsf{4}\}^\infty$. The limit space is the compact Hausdorff space shown in Figure~\subref*{subfig:Vicsek}, which is known as the \newword{Vicsek fractal}.  The gluing relation $\sim$ on $\Omega$ is given by
\[
e\mathsf{0}\overline{\mathsf{3}} \;\sim\; e\mathsf{1}\overline{\mathsf{3}} \qquad\text{and}\qquad e \mathsf{1}\overline{\mathsf{0}} \;\sim\; e \mathsf{2}\overline{\mathsf{0}} \;\sim\; e \mathsf{3}\overline{\mathsf{0}} \;\sim\; e \mathsf{4}\overline{\mathsf{0}}
\]
for every edge $e = \edge_0\cdots\edge_n$ in $G_n$, and also $\mathsf{T}\overline{\mathsf{0}} \sim \mathsf{L}\overline{\mathsf{0}} \sim \mathsf{R}\overline{\mathsf{0}} \sim \mathsf{B}\overline{\mathsf{0}}$.
\end{example}

\subsection{Cells and Rearrangements}
\label{subsec:cellrearrange}

Let $\R=(G_0,e\to R)$ be an expanding replacement system, let $\{G_n\}$ be the corresponding full expansion sequence, and let $X=\Omega/{\sim}$ be the resulting limit space.

\begin{definition}If $e =\edge_0\cdots\edge_n$ is an edge of $G_n$, let $\Omega(e)$ denote the set of all points in $\Omega$ that have $\edge_0\cdots\edge_n$ as a prefix.  The \newword{cell} $C(e)$ is the image of $\Omega(e)$ in the limit space~$X$.
\end{definition}

\begin{figure}
\vgraphics{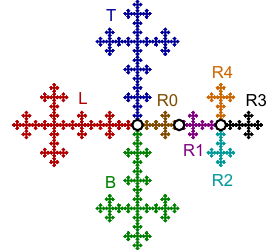}
\qquad\vgraphics{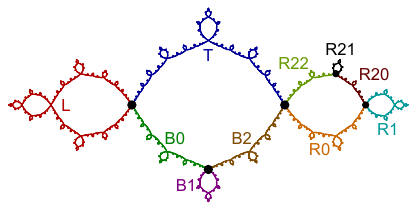}
\caption{Cells for the Vicsek fractal and the basilica.   Each cell $C(\edge_0\cdots\edge_n)$ is labeled by its address~$\edge_0\cdots\edge_n$.}
\label{fig:Cells}
\end{figure}
For example, Figure~\ref{fig:Cells} shows several cells in the Vicsek fractal and in the basilica.

Each cell $C(e)$ has either one or two \newword{boundary points}, namely the gluing vertices that are the endpoints of the edge~$e$.  The complement of the boundary points is the \newword{interior} of the cell, which may or may not be the same as the topological interior.  Note that each cell $C(e)$ is compact, being the image of the compact set~$\Omega(e)$.

The cells of $X$ have the structure of a rooted tree under inclusion, corresponding to the tree of edges under the prefix relation.  Specifically, $C(e) \supseteq C(e')$ whenever $e$ is a prefix of~$e'$, and $C(e)$ and $C(e')$ have disjoint interiors if neither $e$ nor $e'$ is a prefix of the other (see Proposition~\ref{prop:contain}.)  The root of the tree of cells is the whole space~$X$, which corresponds to the empty sequence in the tree of edges.  Note that $X$ is not itself a cell unless the base graph has only one edge.

There is a \newword{canonical homeomorphism} between any two cells of the same type.  More precisely, let $C(e)$ and $C(e')$ be cells of~$X$, where $e$ and $e'$ are either both loops or both not loops, and define a homeomorphism $\Phi\colon \Omega(e)\to\Omega(e')$ by
\[
\Phi(e\zeta_1 \zeta_2 \cdots) \,=\, e' \zeta_1 \zeta_2\cdots
\]
for any edges $\zeta_1,\zeta_2,\ldots$ in~$R$.  Then $\Phi$ descends to a canonical homeomorphism $\phi\colon C(e)\to C(e')$.   Note that the set of canonical homeomorphisms is closed under inverses, composition, and restriction to subcells.\filbreak  

\begin{definition}\label{def:Rearrangement}\quad
\begin{enumerate}
\item A \newword{cellular partition} of $X$ is a cover of $X$ by finitely many cells whose interiors are disjoint.
\item A homeomorphism $f\colon X\to X$ is called a \newword{rearrangement} of $X$ if there exists a cellular partition $\mathcal{P}$ of $X$ such that $f$ restricts to a canonical homeomorphism on each cell of $\mathcal{P}$.
\end{enumerate}
\end{definition}

\begin{figure}
\centering
\vgraphics{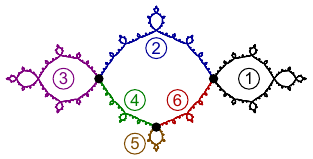}\;\;\longvarrow\;\;\vgraphics{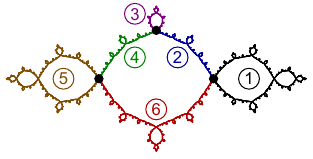}
\caption{A rearrangement of the basilica.  Each of the numbered cells on the left maps to the corresponding cell on the right via a canonical homeomorphism.}
\label{Fig:BasilicaRearrangement}
\end{figure}
\begin{figure}
\centering
\vgraphics{Vicsek1}\qquad\longvarrow\qquad\vgraphics{Vicsek2}
\caption{A rearrangement of the Vicsek fractal.  Each numbered cell on the left maps to the corresponding cell on the right via a canonical homeomorphism.}
\label{Fig:VicsekRearrangement}
\end{figure}
For example, Figure~\ref{Fig:BasilicaRearrangement} shows a rearrangement of the basilica, and Figure~\ref{Fig:VicsekRearrangement} shows a rearrangement of the Vicsek fractal.

\begin{remark}\label{rem:structure}Though we use the phrase ``rearrangements of $X$'', the topological structure of $X$ is not sufficient in general to determine whether a homeomorphism is a rearrangement.  Instead, one must take into account the additional structure that $X$ inherits from its construction as a limit space. To be precise, a limit space should be thought of as a triple $(X,\mathcal{C},\alpha)$, where $X$ is the space itself, $\mathcal{C}$ is the family of cells in~$X$, and $\alpha$ is the function that assigns a finite address $\edge_0\edge_1\cdots\edge_n$ to each cell.
\end{remark}

\begin{proposition}The rearrangements of $X$ form a group under composition.
\end{proposition}
\begin{proof}Clearly the identity homeomorphism is a rearrangement.  For inverses, suppose that $f$ is a rearrangement of~$X$, and let $\mathcal{P}$ be a cellular partition of $X$ such that $f$ restricts to a canonical homeomorphism on each cell of~$\mathcal{P}$.  Then the image
\[
f(\mathcal{P}) \,=\, \{f(C) \mid C\in\mathcal{P}\}
\]
is also a cellular partition of~$X$, and $f^{-1}$ restricts to a canonical homeomorphism on each cell of~$f(\mathcal{P})$, so $f^{-1}$ is a rearrangement.

Finally, suppose that $f$ and $g$ are rearrangements of~$X$.  Let $\mathcal{P}_1$ and $\mathcal{P}_2$ be cellular partitions of $X$ so that $f$ restricts to a canonical homeomorphism on each cell of~$\mathcal{P}_1$, and $g^{-1}$ restricts to a canonical homeomorphism on each cell of~$\mathcal{P}_2$.  Let $\mathcal{Q}$ be the least common refinement of $\mathcal{P}_1$ and $\mathcal{P}_2$, i.e.~the set of all cells in $\mathcal{P}_1\cup\mathcal{P}_2$ that are not properly contained in other cells of~$\mathcal{P}_1\cup\mathcal{P}_2$. Then both $f$ and $g^{-1}$ restrict to a canonical homeomorphism on each cell of~$\mathcal{Q}$, so $f\circ g$ restricts to a canonical homeomorphism on each cell of~$g^{-1}(\mathcal{Q})$.
\end{proof}

We refer to the group of all rearrangements of the limit space $X$ as the \newword{rearrangement group} of~$X$. Subsection~\ref{subsec:examples} discusses many examples of rearrangement groups. The interested reader may wish to skip ahead to that subsection to see some of these examples before continuing to Subsection \ref{subsec:graphpair}.

\subsection{Graph Pair Diagrams}
\label{subsec:graphpair}

In this subsection we introduce graph pair diagrams, which provide a simple graphical representation of rearrangements and their action on the limit space.

Given a replacement system $\R = (G_0,e\to R)$, there is a one-to-one correspondence between expansions of the base graph~$G_0$ and cellular partitions of the corresponding limit space.  In particular, given an expansion $E$ of $G_0$, the edges $e_1,\ldots,e_n$ of $E$ define a cellular partition $\{C(e_1),\ldots,C(e_n)\}$ of the limit space~$X$, and every cellular partition has this form.

\begin{figure}[b]
\centering
\vgraphics{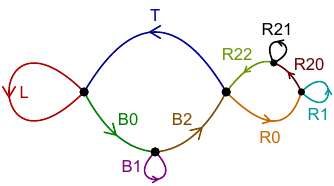}
\hfill\vgraphics{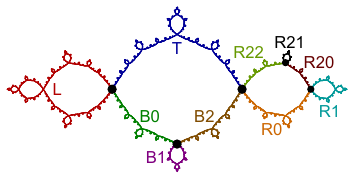}
\caption{An expansion of the base graph for the basilica and the corresponding cellular partition.}
\label{fig:StructureGraphs}
\end{figure}
For example, Figure~\ref{fig:StructureGraphs} shows the cellular partition of the basilica corresponding to a certain expansion of the base graph.  Note that the edges of the expansions intersect in precisely the same way as the cells of the partition, with the vertices of the expansion corresponding to the boundary vertices of cells of the partition.

If $f\colon X\to X$ is a rearrangement that maps the cells of one cellular partition canonically to the cells of another, then $f$ must induce an isomorphism between the corresponding expansions.  This prompts the following definition.

\begin{definition}Let $f\colon X\to X$ be a rearrangement.  A \newword{graph pair diagram} for $f$ is a triple $(E,E',\varphi)$, where $E$ and $E'$ are expansions of $G_0$ and $\varphi\colon E\to E'$ is an isomorphism, such that $f$ maps $C(e)$ canonically to $C(\varphi(e))$ for each edge $e$ in~$E$.
\end{definition}

\begin{figure}
\centering
\footnotesize
$
\renewcommand{\arraystretch}{1.25}
\underset{\textstyle\text{Domain\rule{0pt}{10pt}}}{\vgraphics{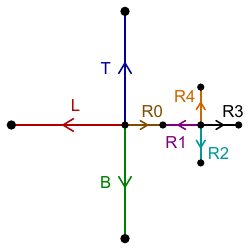}}
\qquad\quad\underset{\textstyle\text{Range\rule{0pt}{10pt}}}{\vgraphics{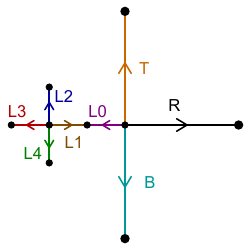}}
\qquad\quad
\scriptstyle\raisebox{-2ex}{$\begin{array}{@{}c@{\;\;\mapsto\;\;}c@{}}
\textsf{L} & \textsf{L3} \\
\textsf{T} & \textsf{L2} \\
\textsf{B} & \textsf{L4} \\
\textsf{R0} & \textsf{L1} \\
\textsf{R1} & \textsf{L0} \\
\textsf{R2} & \textsf{B} \\
\textsf{R3} & \textsf{R} \\
\textsf{R4} & \textsf{T}
\end{array}$}$
\caption{A graph pair diagram for a rearrangement of the Vicsek fractal.  The corresponding rearrangement is shown in Figure~\ref{Fig:VicsekRearrangement}.}
\label{fig:GraphPairDiagram}
\end{figure}
\begin{example}Let $f$ be the rearrangement of the Vicsek fractal shown in Figure~\ref{Fig:VicsekRearrangement}.  One possible graph pair diagram for $f$ is shown in Figure~\ref{fig:GraphPairDiagram}.  In this picture, the domain and range graphs are drawn, and the isomorphism is defined by the table of mappings shown on the right.  (For clarity, we have also colored corresponding edges in the two graphs using the same colors.)

\begin{figure}[tb]
\centering
\includegraphics{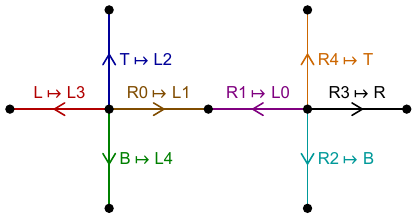}
\caption{Another convention for drawing the graph pair diagram shown in Figure~\ref{fig:GraphPairDiagram}.}
\label{fig:GraphPairDiagram2}
\end{figure}
Figure~\ref{fig:GraphPairDiagram2} shows another drawing of this same graph pair diagram.  Instead of showing separate copies of the isomorphic graphs $E$ and~$E'$, this picture shows only a single graph with two sets of labels.  This convention for drawing graph pair diagrams is more compact, but conveys less geometric intuition for how the corresponding rearrangement acts on the limit space.
\end{example}

\begin{figure}[tb]
\centering
\includegraphics{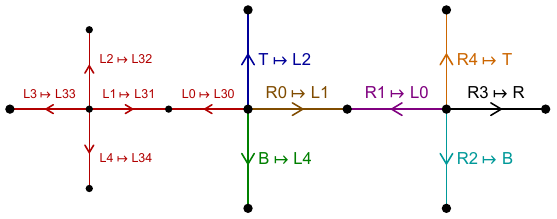}
\caption{An unreduced graph pair diagram for the rearrangement from Figure~\ref{Fig:VicsekRearrangement}.}
\label{fig:GraphPairDiagramUnreduced}
\end{figure}
The graph pair diagram for a rearrangement is not unique.  For example, Figure~\ref{fig:GraphPairDiagramUnreduced} shows a different graph pair diagram for the rearrangement from the last example.  In this diagram, the $\mathsf{L}$ edge has been expanded in the domain graph, and the corresponding $\mathsf{L3}$ edge has been expanded in the range graph.  Thus the five leftmost edges in Figure~\ref{fig:GraphPairDiagramUnreduced} all correspond to portions of the cell described by the leftmost edge of Figure~\ref{fig:GraphPairDiagram2}.

In general, if $(E,E',\varphi)$ is a graph pair diagram and $e$ is an edge of~$E$, then $(E \exp e, E' \exp \varphi(e), \varphi')$ is another graph pair diagram for the same rearrangement, where $\varphi'\colon E\exp e \to E'\exp \varphi(e)$ is the isomorphism that agrees with $\varphi$ on $E-\{e\}$ and maps $e\edge$ to $\varphi(e)\edge$ for every edge $\edge$ in~$R$.  We say that a graph pair diagram is \newword{reduced} if it cannot be obtained from a smaller graph pair diagram in this fashion. For example, the graph pair diagram in Figure~\ref{fig:GraphPairDiagram2} is reduced, but the one in Figure~\ref{fig:GraphPairDiagramUnreduced} is not.

\begin{proposition}\label{prop:UniqueReduced} Every rearrangement has a unique reduced graph pair diagram.
\end{proposition}
\begin{proof}Let $f\colon X\to X$ be a rearrangement.  We say that $f$ is \textit{regular} on a cell $C$ if $f$ restricts to a canonical homeomorphism on~$C$.  Then an expansion $E$ is the domain graph of a graph pair diagram for $f$ if and only if $f$ is regular on $C(e)$ for each edge $e$ of~$E$.  It follows that a graph pair diagram $(E,E',\varphi)$ is reduced if and only if the edges of $E$ correspond precisely to the maximal cells on which $f$ is regular.
\end{proof}

\begin{remark}Graph pair diagrams can be thought of as an analogue of the tree pair diagrams for elements of $F$, $T$, and $V$~(see~\cite{CFP}).  In particular, recall that the cells in $X$ have the structure of an infinite tree under inclusion, where the root is not itself a cell.  There is a one-to-one correspondence between
\begin{itemize}
\item Cellular partitions of $X$,
\item Expansions of $G_0$, and
\item Finite rooted subtrees of the tree of cells.
\end{itemize}
In particular, every rearrangement can be described by a pair of finite rooted trees.  For Thompson's groups $F$, $T$, and $V$, this gives the tree pair diagram for a rearrangement.  However, tree pairs are not as useful for other rearrangement groups, since it is not possible to see from the respective trees whether the corresponding expansions are isomorphic.
\end{remark}

\begin{remark}It is possible to compose two rearrangements directly from the graph pair diagram.  Specifically, let $f_1$ and $f_2$ be rearrangements, and let $(E_1,E_1',\varphi_1)$ and $(E_2,E_2',\varphi_2)$ be a corresponding pair of graph pair diagrams.  By expanding if necessary, we may assume that $E_1' = E_2$.  Then the composition $f_2\circ f_1$ is the rearrangement whose graph pair diagram is $(E_1,E_2',\varphi_2\circ\varphi)$.
\end{remark}

\subsection{The Topology of the Limit Space}
\label{subsec:top}

In this subsection, we prove several important technical statements regarding the gluing relation, the limit space, cells, and canonical homeomorphisms, most of which were stated without proof in Section~\ref{sec:rearrangement}.

Let $\R = (G_0,e\to R)$ be a replacement system, which we assume to be expanding (see Definition~\ref{def:expanding}).  Let $\{G_n\}$ be the corresponding full expansion sequence, let $\Omega$ be the resulting symbol space, let $\sim$ be the gluing relation on~$\Omega$, and let $X = \Omega/\sim$ be the resulting limit space.  We begin by characterizing the gluing relation~$\sim$ in terms of the gluing vertices.

\begin{proposition}\label{prop:GluingRelation}\quad
\begin{enumerate}
\item Each gluing vertex is represented by at least one point in\/~$\Omega$.
\item Each point in\/ $\Omega$ represents at most one gluing vertex.
\item Two points in\/ $\Omega$ are equivalent under the gluing relation if and only if they represent the same gluing vertex.
\item The gluing relation~$\sim$ is an equivalence relation.
\end{enumerate}
\end{proposition}
\begin{proof} For (1), let $v \in V(G_n)$ be a gluing vertex.  Since $R$ is expanding, $G_n$ has no isolated vertices, so there exists an edge $\edge_0\cdots\edge_n$ in $G_n$ that is incident on~$v$.  Since neither the initial nor terminal vertex of $R$ is isolated, we can inductively choose edges $\edge_k$ in $R$ for $k>n$ so that $\edge_0\cdots\edge_k$ is incident on~$v$.  Then $\edge_0\edge_1\edge_2\cdots$ is a point in $\Omega$ that represents~$v$.

For (2), let $\edge_0\edge_1\edge_2\cdots$ be a point in~$\Omega$, and let $u$ and $v$ be distinct gluing vertices. Suppose that $\edge_0\cdots\edge_n$ is incident on both $u$ and $v$ in~$G_n$.  Since $R$ is expanding, the initial and terminal vertices of $R$ are not connected by an edge, so $\edge_0\cdots\edge_n\edge_{n+1}$ cannot be incident on both $u$ and $v$ in~$G_{n+1}$.  Hence $\edge_0\edge_1\edge_2\cdots$ cannot represent both $u$ and~$v$.

The backward direction of (3) is clear.  For the forward direction, let $\edge_0\edge_1\edge_2\cdots$ and $\edgetwo_0\edgetwo_1\edgetwo_2\cdots$ be distinct points of $\Omega$ that are equivalent under the gluing relation.  Then there exists an $n$ so that $\edge_0\cdots\edge_n$ and $\edgetwo_0\cdots\edgetwo_n$ are distinct edges in~$G_n$. Since $R$ is expanding, the edges $\edge_0\cdots\edge_{n+1}$ and $\edgetwo_0\cdots\edgetwo_{n+1}$ in $G_{n+1}$ share at most one vertex~$v$.  Then $\edge_0\cdots\edge_k$ and $\edgetwo_0\cdots\edgetwo_k$ must both be incident on $v$ for all $k > n$, so $\edge_0\edge_1\edge_2\cdots$ and $\edgetwo_0\edgetwo_1\edgetwo_2\cdots$ both represent~$v$.

Statement (4) follows immediately from statements (2) and (3).
\end{proof}

\begin{figure}
\centering
\vgraphics{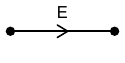}
\qquad\qquad
\vgraphics{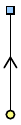} \quad\longvarrow\quad \vgraphics{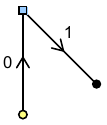}
\caption{A non-expanding replacement system.}
\label{fig:NonExpanding}
\end{figure}
\begin{remark}The gluing relation $\sim$ is not necessarily an equivalence relation in the case of a non-expanding replacement system.  For example, if $\R$ is the replacement system shown in Figure~\ref{fig:NonExpanding}, then $\mathsf{E}\overline{\mathsf{0}} \sim \mathsf{E1}\overline{\mathsf{0}}$ and $\mathsf{E1}\overline{\mathsf{0}} \sim \mathsf{E11}\overline{\mathsf{0}}$, but $\mathsf{E}\overline{\mathsf{0}} \not\sim \mathsf{E11}\overline{\mathsf{0}}$, so the gluing relation is not transitive.

Of course, one could still define the limit space in this case by using the transitive closure of the gluing relation.  However, this often results in a limit space that is not Hausdorff.  For example, if we let $\approx$ be the transitive closure of the gluing relation for the replacement system in Figure~\ref{fig:NonExpanding}, then $\mathsf{E}\overline{\mathsf{1}} \not\approx \mathsf{E}\overline{\mathsf{0}}$, but every neighborhood of $\mathsf{E}\overline{\mathsf{1}}$ in $\Omega$ contains a point from the $\approx$-equivalence class of~$\mathsf{E}\overline{\mathsf{0}}$, e.g.~$\mathsf{E11}\cdots\mathsf{1}\overline{0}$ for a sufficiently long string of $\mathsf{1}$'s.
\end{remark}

We next prove some basic facts about cells.  For convenience, we shall henceforth refer to points in $X$ that are not gluing vertices as \newword{regular points}.  Note that each regular point has a single representative in~$\Omega$, and that every regular point contained in a cell is an interior point of that cell.

\begin{proposition}\label{prop:contain}\quad
\begin{enumerate}
\item The interior of each cell in $X$ has at least one gluing vertex and one regular point.
\item If $p$ is a point in the interior of a cell $C(e)$, then every address for $p$ has $e$ as a prefix.
\item If $e$ is a prefix for $f$, then $C(e) \supseteq C(f)$.
\item If neither $e$ nor $f$ is a prefix for the other, then the interiors of $C(e)$ and $C(f)$ are disjoint.
\end{enumerate}
\end{proposition}
\begin{proof}For statement (1), let $C(e)$ be a cell in $X$. Since $\R$ is expanding, the replacement graph $R$ has an interior vertex $\nu$, so $e\nu$ is a gluing vertex that lies in the interior~$C(e)$.  Similarly, $R$ must have an edge $\edgetwo_i$ that is incident on the initial vertex, and an edge $\edgetwo_t$ that is incident on the terminal vertex.  If $\edgetwo_i$ is incoming at the initial vertex of~$R$, then $\edge_0\cdots \edge_n \overline{\edgetwo_i}$ is a regular point in the interior of~$C(e)$, while if $\edgetwo_i$ is outgoing at the initial vertex of~$R$, then $\edge_0\cdots \edge_n\overline{\edgetwo_i\edgetwo_t}$ is a regular point in the interior of $C(e)$.

Statement (2) is obvious if $p$ is a regular point.  If $p$ is a gluing vertex $\edge_0\edge_1 \ldots \edge_i\nu$, then $e$ must be a prefix for $\edge_0\edge_1\cdots \edge_i$.  Then any edge in any $G_n$ incident on $p$ must also have $e$ as a prefix, and the statement follows.

Finally, statement (3) follows immediately from the definition of the cells, and statement (4) follows from statement~(2).
\end{proof}

We wish to prove the following theorem.

\begin{theorem}\label{thm:hausdorff} $X$ is a compact metrizable space.
\end{theorem}

Since $\Omega$ is compact and metrizable and $X$ is a quotient of~$\Omega$, it suffices to prove that $X$ is Hausdorff (see~\cite[Proposition~IX.17]{Bou}). To prove this, we need some open and closed sets in~$X$.

\begin{lemma}\label{lem:celltop}\quad
\begin{enumerate}
\item One-point sets are closed in~$X$.
\item Each cell is closed in~$X$, and the interior of each cell is open in~$X$.
\end{enumerate}
\end{lemma}
\begin{proof}Let $q\colon \Omega\to X$ be the quotient map.  Statement (1) is obvious for regular points, since the preimage of a regular point is a single point in~$\Omega$.  For a gluing vertex $v$ with address $\edge_0\cdots \edge_n \nu$, observe that
\[
q^{-1}(v) \;=\; \bigcap_{i>n}\, \bigcup_{e\in E_i} \Omega(e),
\]
where $E_i$ is the set of edges in $G_i$ that are incident on $v$.  Since each $\Omega(e)$ is closed and each $E_i$ is finite, this set is closed.

For statement (2) if $C(e)$ is a cell with boundary vertices $v$ and $w$, then by statement (1) of Proposition~\ref{prop:contain}
\[
q^{-1}\bigl(C(e)\bigr) \;=\; \Omega(e) \cup q^{-1}(v) \cup q^{-1}(w),
\]
which is closed since $\Omega(e)$, $q^{-1}(v)$, and $q^{-1}(w)$ are closed.  Similarly, the preimage of the interior of $C(e)$ is
\[
\Omega(e) - \bigl(q^{-1}(v) \cup q^{-1}(w)\bigr),
\]
which is open since $\Omega(e)$ is open and $q^{-1}(v)$ and $q^{-1}(w)$ are closed.
\end{proof}

\begin{proof}[Proof of Theorem~\ref{thm:hausdorff}] For each gluing vertex $v\in V(G_n)$, let $\mathrm{St}_n(v)$ denote the union of $\{v\}$ with the interiors of the cells corresponding to the edges of $G_n$ that are incident on~$v$. Note that $\mathrm{St}_n(v)$ is open, since its complement is the union of the cells corresponding to the other edges of~$G_n$.

Now let $p$ and $q$ be distinct points of $X$. There are three cases to consider.
\begin{enumerate}
\item If $p$ and $q$ are regular points, then they have distinct addresses $\{ \edge_n \}$ and $\{ \edgetwo_n \}$ respectively.  Suppose that $\edge_i \not = \edgetwo_i$ for some $i$.  Then the interiors of the cells $C(\edge_0\edge_1 \ldots \edge_i)$ and $C(\edgetwo_0\edgetwo_1 \ldots \edgetwo_i)$ are disjoint open sets containing $p$ and $q$.
\item If both $p$ and $q$ are gluing vertices, then let $i$ be the minimal value so that $p, q \in V(G_i)$.  Since the replacement system is expanding, there are no edges in $G_{i+1}$ that are incident to both $p$ and $q$.  Then $\mathrm{St}_{i+1}(p)$ and $\mathrm{St}_{i+1}(q)$ are disjoint.
\item If $p$ is a gluing vertex but $q$ is not,
let $\{ \edgetwo_n \}$ be the address for $q$.  Then for sufficiently large $i$, the gluing vertex $p \in V(G_i)$ and $\edgetwo_i$ is not incident to~$p$.  Then $\mathrm{St}_i(p)$ and the interior of the cell $C(\edgetwo_0\edgetwo_1 \ldots \edgetwo_i)$ are disjoint open sets containing $p$ and $q$.\qedhere
\end{enumerate}
\end{proof}

\begin{remark}The connectedness of $R$ determines the connectedness of $X$.  In particular:
\begin{enumerate}
\item If $R$ is connected, then each cell in $X$ is path connected.  Conversely, if $R$ is disconnected, then each cell in $X$ has infinitely many components.
\item The initial and terminal vertices of $R$ lie in different components if and only if $X$ is totally disconnected.
\end{enumerate}
\end{remark}

\subsection{Metrics on the Limit Space}
\label{subsec:metrics}

In this subsection, we define a natural family of metrics on the limit space in the connected case.  Again, let $\R = (G_0,e\to R)$ be an expanding replacement system, and let $X$ be the corresponding limit space.  We assume that $R$ and $G_0$ are connected, and therefore $X$ is connected as well. 

We begin by fixing geodesic metrics on $G_0$ and~$R$.  Since these are graphs, this is equivalent to assigning a positive length $\ell(\edge)$ to each edge $\edge$ of $R$ or~$G_0$.  We make two assumptions about the metric on~$R$:
\begin{enumerate}
\item The distance between the boundary vertices of $R$ is equal to~$1$.
\item Every edge of $R$ has length strictly less than~$1$.
\end{enumerate}
Let $\{G_n\}$ be the full expansion sequence for~$\R$.  For each edge $\edge_0\edge_1\cdots \edge_n$ of~$G_n$, where $\edge_0$ is an edge of $G_0$ and $\edge_1,\ldots,\edge_n$ are edges of~$R$, we assign a length according to the formula
\[
\ell(\edge_0\edge_1\cdots\edge_n) = \ell(\edge_0)\,\ell(\edge_1)\,\cdots\,\ell(\edge_n).
\]
This puts a geodesic metric on each of the graphs~$G_n$.

Now, recall that the vertex sets $V(G_n)$ are nested, with
\[
V(G_0) \subset V(G_1) \subset V(G_2) \subset \cdots.
\]
Since the distance between the initial and terminal vertices of $R$ is equal to~$1$, the distance between any two vertices in $G_n$ is the same as the distance between them in $G_{n+1}$; that is, the metrics on all of the $G_n$'s agree for the vertices.  This gives us a metric $d$ on the set $V$ of gluing vertices in~$X$.

We wish to extend $d$ to all of~$X$.  Note that $V$ is dense in~$X$, since the preimage of $V$ is dense in the symbol space.  To prove that $d$ extends, we first need an upper bound on the diameters of cells.

For the following lemma, recall that the \newword{diameter} of any set $S\subset V$ is defined by the formula
\[
\mathrm{diam}(S) = \sup_{v,w\in S} d(v,w).
\]

\begin{lemma}\label{lem:diameterbound}There exists a constant $k>0$ with the following property: for every $n\in\mathbb{N}$ and every edge $e$ in $G_n$,
\[
\mathrm{diam}\bigl(V\cap C(e)\bigr) \leq k\,\ell(e).
\]
\end{lemma}
\begin{proof}Let $r$ be the maximum length of any edge in~$R$, and let $M = \mathrm{diam}(R)$.  We will prove the proposition for $k=1 + 2Mr/(1-r)$.

Let $e$ be an edge in some $G_n$.  Let $H_0$ be the subgraph of $G_n$ consisting of $e$ and its two endpoints, and let $\{H_i\}$ be the full expansion sequence for~$H_0$.  Note that each $H_i$ is a subgraph of $G_{n+i}$, with $V(H_i) = V(G_{n+i})\cap C(e)$.  Each edge of $H_i$ has length at most $r^{i}\ell(e)$, so each copy of $R$ that we paste into $H_i$ to obtain $H_{i+1}$ has diameter at most $Mr^i\ell(e)$.  It follows that
\[
\mathrm{diam}(H_{i+1}) \leq \mathrm{diam}(H_i) + 2Mr^i\,\ell(e).
\]
Since $\mathrm{diam}(H_0) \leq \ell(e)$, we conclude that
\[
\mathrm{diam}(H_i) \leq \ell(e) + 2M(r+r^2+r^3 + \cdots)\,\ell(e) = k\,\ell(e)
\]
for all $i$.  In particular $\mathrm{diam}\bigl(V(G_{n+i})\cap C(e)\bigr)\leq k\,\ell(e)$ for all~$i$, and therefore $\mathrm{diam}\bigl(V\cap C(e)\bigr)\leq k\,\ell(e)$.
\end{proof}

\begin{proposition}The metric on $V$ extends to a continuous metric on~$X$.
\end{proposition}
\begin{proof}In general, if $X$ is a topological space and $d$ is a metric defined on a dense subset $V$ of $X$, then $d$ can be extended to a continuous function $X\times X\to\mathbb{R}$ if and only if the following condition holds:
\begin{quote}
For every point $p \in X$, there exists a sequence $\{U_i\}$ of open neighborhoods of $p$ such that $\mathrm{diam}(V\cap U_i) \to 0$.
\end{quote}
We will verify this condition using two cases.  First, suppose $p$ is a regular point of~$X$ with address $\{\edge_i\}$.  Then the interior of each of the cells $C(\edge_0\edge_1\cdots\edge_i)$ is an open set containing~$p$.  By Lemma~\ref{lem:diameterbound}, each of these sets has diameter at most $k\,\ell(\edge_0\edge_1\cdots\edge_i)\leq kr^i\ell(\edge_0)$, where $r$ is the maximum length of any edge in~$R$, so the diameter approaches zero as~$i\to\infty$.  The second case is when $p$ is a gluing vertex, say $p\in V(G_n)$ for some~$n$.  For this case, the sequence $\mathrm{St}_{n+i}(p)$ of open stars defined in the proof of Theorem~\ref{thm:hausdorff} has the property that $\mathrm{diam}\bigl(\mathrm{St}_{n+i}(p)\bigr) \leq 2kr^iL$ for all~$i$, where $L$ is the maximum length of all edges incident on $p$ in~$G_n$.  In particular, the diameters of these stars approaches zero as~$i\to\infty$.

Thus $d$ extends to a continuous function $d\colon X\times X\to\mathbb{R}$, which is necessarily a pseudometric.  To prove that $d$ is a metric, let $p,q$ be distinct points in~$X$ with addresses $\{\edge_n\}$ and $\{\edgetwo_n\}$.  Since $p$ and $q$ are distinct, there exists an $n$ such that the edges $\edge_0\edge_1\cdots\edge_n$ and $\edgetwo_0\edgetwo_1\cdots\edgetwo_n$ do not share a vertex in~$G_n$.  Let $\delta$ denote the distance between these edges in~$G_n$. Let $S = V\cap C(\edge_0\edge_1\cdots\edge_n)$ and $T = V\cap C(\edgetwo_0\edgetwo_1\cdots\edgetwo_n)$.  Note that $S$ is dense in $C(\edge_0 \edge_1\cdots\edge_n)$ and $T$ is dense in $C(\edgetwo_0\edgetwo_1\cdots\edgetwo_n)$, so in particular $p$ lies in the closure of $S$ and $q$ lies in the closure of~$T$.  But every point in $S$ lies a distance of at least $\delta$ from every point in $T$, and therefore $d(p,q)\geq \delta$.
\end{proof}

\begin{notes}\quad
\begin{enumerate}
\item Since $X$ is compact and Hausdorff, all continuous metrics on $X$ define the same topology.  Thus the topology determined by the metric we have described is the same as the previously defined metrizable topology on~$X$.
\item Though we will not prove it here, the metric that we have defined on $X$ is actually a geodesic metric.
\item With respect to this metric, the canonical homeomorphism between any two cells in $X$ having the same number of boundary vertices is a similitude, so rearrangements are piecewise-similar homeomorphisms.
\end{enumerate}
\end{notes}

\section{Rearrangement Groups}
\label{sec:RearrangementGroups}

In this section we initiate the algebraic study of rearrangement groups.  Subsection~\ref{subsec:examples} presents a large number of examples of rearrangement groups acting on self-similar spaces.  In Subsection~\ref{subsec:thompson} we show that Thompson's groups $F$, $T$, and $V$ are rearrangement groups, as are many other generalized Thompson groups.  Further, we show that a large class of rearrangement groups contain a copy of Thompson's group~$F$ (Proposition~\ref{prop:containsf}), and that every rearrangement group can be embedded into Thompson's group~$V$. We characterize all finite subgroups of rearrangement groups in Subsection~\ref{subsec:finite}, and we use this characterization to prove that many of the rearrangement groups under consideration are non-isomorphic.  Lastly, in Subsection~\ref{subsec:colors}, we examine the natural generalization of replacement systems to graphs with edge colorings.  This generalization allows us to build colored replacement systems whose limit spaces are a broader class of fractals, as well as recover many diagram groups as rearrangement groups.

\subsection{Examples}
\label{subsec:examples}

In this subsection we present several examples of replacement systems and limit sets, and discuss the corresponding rearrangement groups.  The examples introduced here include infinite families of replacement systems that generalize those for the Vicsek fractal and basilica Julia set, as well as an example of a replacement system for the Julia set of a rational map and a replacement system whose corresponding rearrangement group is trivial.

\begin{table}
\centering
\renewcommand{\arraystretch}{1.25}
\newcommand{\mystrut}{\rule{0pt}{0.7in}}
\begin{tabular}{@{}c@{\hspace{0.35in}}c@{\hspace{0.45in}}l@{\hspace{0.25in}}c@{}}
\hline
Index & Base Graph & \quad\quad\quad Rule & Limit Space \\
\hline
$n=3$ & \mystrut\vgraphics{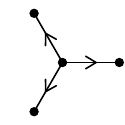} & $\vgraphics{EdgeReplacement1Tall} \;\;\longvarrow\;\; \vgraphics{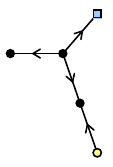}$ &
\vgraphics{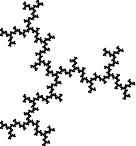} \\
$n=4$ & \mystrut\vgraphics{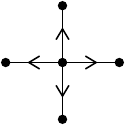} & $\vgraphics{EdgeReplacement1Tall} \;\;\longvarrow\;\; \vgraphics{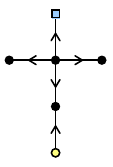}$ &
\vgraphics{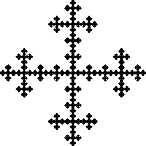} \\
$n=5$ & \mystrut\vgraphics{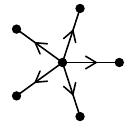} & $\vgraphics{EdgeReplacement1Tall} \;\;\longvarrow\;\; \vgraphics{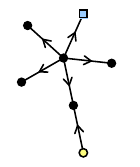}$ &
\vgraphics{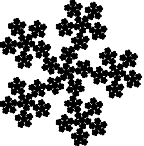} \\
\vdots & \rule{0pt}{24pt}\vdots & \hspace{0.55in}\vdots & \vdots \\[12pt]
\hline
\end{tabular}
\caption{The Vicsek family of replacement systems.}
\label{tab:VicsekFamily}
\end{table}
\begin{example}[The Vicsek Family]
\label{ex:VicsekFamily}
Table~\ref{tab:VicsekFamily} shows the \newword{Vicsek family} of replacement systems, of which the Vicsek replacement system from Example~\ref{ex:Vicsek} is the $n=4$ case.  Each of the resulting limit spaces is homeomorphic to the standard universal dendrite of order~$n$ (see~\cite{Ch}).  These spaces can also be realized as the fixed sets of iterated function systems in the plane, or as Julia sets associated to quadratic polynomials.  In the former case, the rearrangement groups act by homeomorphisms that are piecewise Euclidean similarities, while in the latter case they act by piecewise conformal homeomorphisms.

The Vicsek family of rearrangement groups are nested, with the $n=3$ rearrangement group contained in the rearrangement group of the Vicsek fractal, which is in turn contained in the rearrangement group of the $n=5$ case.  The finite subgroups of these rearrangement groups are examined in detail in Subsection~\ref{subsec:finite}. We will prove in Theorem~\ref{thm:vicsekfinfty} that all of these rearrangement groups have type~$F_\infty$.
\end{example}

\begin{figure}[b]
\centering
\subfloat[\label{subfig:MandelbrotBasilica}]{\includegraphics{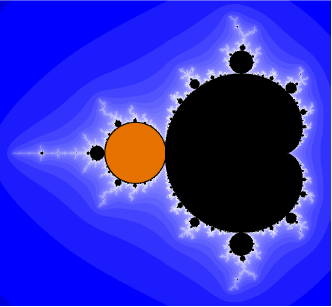}}
\hfill
\subfloat[\label{subfig:MandelbrotRabbits}]{\includegraphics{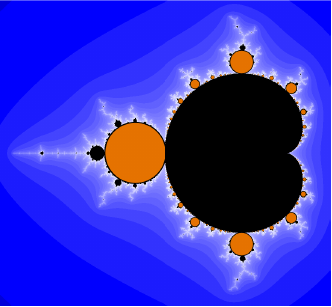}}
\caption{(a) The interior component of the Mandelbrot set that corresponds to the basilica. (b) The interior components of the Mandelbrot set that correspond to the family of rabbits.}
\end{figure}
\begin{example}[The Basilica Thompson Group]
The rearrangement group $T_B$ for the basilica replacement system from Example~\ref{ex:BasilicaReplacement} is called the \newword{basilica Thompson group}. In \cite{BeFo} the authors proved the following facts about this group:
\begin{enumerate}
\item $T_B$ is generated by four elements.
\item Thompson's group $T$ contains copies of $T_B$, and $T_B$ contains $T$.
\item The commutator subgroup $[T_B,T_B]$ has index two in $T_B$ and is simple.  It is not isomorphic to~$T$.
\end{enumerate}
More recently, Witzel and Zaremsky show that $T_B$ is not finitely presented~\cite{WiZa}.  For further discussion, see Example~\ref{ex:BasilicanotFinfty}.

As mentioned previously, the limit space $X$ for the basilica replacement system is homeomorphic to the Julia set for the function $f(z) = z^2-1$.  More generally, Figure~\subref*{subfig:MandelbrotBasilica} shows the interior component of the Mandelbrot set that contains~$-1$.  If $c$ is any point in this component, then the Julia set $J_c$ for $f(z) = z^2+c$ is homeomorphic to the basilica.  For any such~$J_c$, the canonical homeomorphisms between cells of $X$ act as conformal homeomorphisms on pieces of~$J_c$, so $T_B$ acts by piecewise-conformal homeomorphisms on~$J_c$.
\end{example}

\begin{table}
\centering
\renewcommand{\arraystretch}{1.25}
\newcommand{\mystrut}{\rule{0pt}{0.7in}}
\begin{tabular}{@{}c@{\hspace{0.35in}}c@{\hspace{0.5in}}l@{\hspace{0.05in}}c@{}}
\hline
Index & Base Graph & \quad\; Rule & Limit Space \\
\hline
$n=1$ & \mystrut\vgraphics{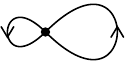} & $\vgraphics{EdgeReplacement1} \;\;\longvarrow \vgraphics{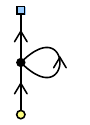}$ &
\vgraphics{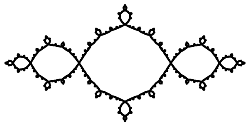} \\
$n=2$ & \mystrut\vgraphics{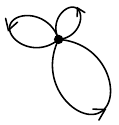} & $\vgraphics{EdgeReplacement1} \;\;\longvarrow \vgraphics{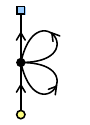}$ &
\vgraphics{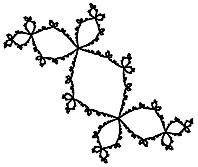} \\
$n=3$ & \mystrut\vgraphics{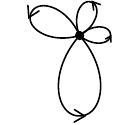} & $\vgraphics{EdgeReplacement1} \;\;\longvarrow \vgraphics{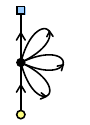}$ &
\vgraphics{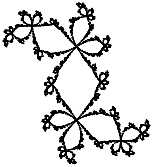} \\
\vdots & \rule{0pt}{24pt}\hspace{0.3in}\vdots & \hspace{0.3in}\vdots & \vdots \\[12pt]
\hline
\end{tabular}
\caption{Replacement systems for the rabbit family. Note that while the base graph for $n=1$ is different than the \protect one for the basilica replacement system given in Figure~\protect\subref*{subfig:BasilicaReplacementRule}, the rearrangement group is isomorphic, as we will show in Subsection~\ref{subsec:genrearrange}.}
\label{tab:Rabbits}
\end{table}
\begin{example}[The Family of Rabbits]
\label{ex:rabbit}
The basilica replacement system can be generalized to the family of \newword{rabbit replacement systems}, shown in Table~\ref{tab:Rabbits}.  These correspond to Julia sets for functions of the form  $f(z) = z^2+c$, where $c$ lies in any interior component of the Mandelbrot set that is adjacent to the main cardioid, as shown in Figure~\subref*{subfig:MandelbrotRabbits}. There is one rabbit replacement system for each natural number~$n$, where $n=1$ corresponds to the basilica, $n=2$ is corresponds to the well-known Douady rabbit, $n=3$ corresponds to a three-earred rabbit, and so forth.  The $n=0$ case corresponds to Thompson's group~$T$ (see Proposition~\ref{prop:ThompsonsGroups}).

The resulting replacement groups are nested, with $T$ contained in the basilica Thompson group, which is in turn contained in the rearrangement group for the Douady rabbit, and so on. All of these groups are finitely generated (see Example~\ref{ex:BasilicanotFinfty}).
\end{example}

In general, a complex polynomial $f$ is called \newword{postcritically finite} if every critical point of $f$ has finite forward orbit.  The structure of the Julia set for a postcritically finite polynomial can be described combinatorially by its ``Hubbard tree''~\cite{DH}.  Using Hubbard trees, the authors have developed an algorithm that derives replacement systems for the Julia sets of many different postcritically finite polynomials.  In general, the replacement system derived in this fashion involves colored edges and multiple replacement rules as explained in Subsection~\ref{subsec:colors}.

\begin{figure}
\subfloat[\label{subfig:BubbleBath}]{\vgraphics{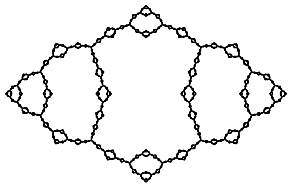}}
\hfill
\subfloat[\label{subfig:BubbleBathReplacement}]{$\vgraphics{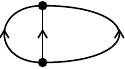}\quad\qquad\vgraphics{EdgeReplacement1} \;\;\longvarrow \vgraphics{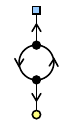}$ }
\caption{(a) The Julia set for $f(z) = z^{-2} -1$. (b) A replacement system whose limit space is homeomorphic to this Julia set.}
\label{fig:BubbleBathAll}
\end{figure}
\begin{example}[Julia Sets for Rational Maps]
Julia sets for rational functions---even postcritically finite ones---cannot be described in general using replacement systems.  For example, J.~Milnor and T.~Lei have proven that there exists a postcritically finite quadratic rational function whose Julia set is homeomorphic to a Sierpi\'{n}ski carpet~\cite{MiLe}.  Since the Sierpi\'{n}ski carpet cannot be disconnected by removing any finite set, it is not homeomorphic to the limit set of any replacement system.

However, there are certainly some rational functions whose Julia sets can be described as the limit sets of edge replacement systems. For example, Figure~\subref*{subfig:BubbleBath} shows the Julia set for the rational function $f(z) = z^{-2}-1$.  This Julia set is homeomorphic to the limit space of the replacement system shown in Figure~\subref*{subfig:BubbleBathReplacement}.  Each of the canonical homeomorphisms between cells acts as a homeomorphism between portions of the Julia set, and hence the rearrangement group acts on the Julia set by piecewise-conformal homeomorphisms.
\end{example}

\begin{figure}
\centering
\subfloat[\label{subfig:TrivialRule}]{\vgraphics{EdgeReplacement1Tall} \quad\longvarrow\quad \vgraphics{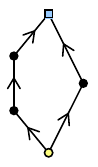}}
\qquad\qquad\qquad
\subfloat[\label{subfig:TrivialLimitSpace}]{\vgraphics{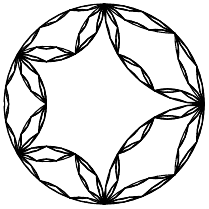}}
\caption{(a) A replacement rule that leads to a trivial replacement group. (b)~The associated limit space, using a single edge as the base graph.}
\end{figure}
\begin{example}[A Trivial Rearrangement Group]
\label{ex:trivial}
Though rearrangement groups provide a wide variety of interesting examples, it is not difficult to construct replacement systems with trivial rearrangement group.  For example, consider the replacement rule shown in Figure~\subref*{subfig:TrivialRule}.  If we use a single edge as the base graph, the resulting limit space is shown in Figure~\subref*{subfig:TrivialLimitSpace}.

Any rearrangement of this limit space must fix the two vertices of~$G_0$, since these are the only source and sink, respectively, in any expansion.  But removing these vertices yields two complementary components that are not homeomorphic, since the left component has two different points whose removal disconnects it into three pieces, while right component has only one such point.  It follows that any rearrangement must fix the vertices of~$G_1$, and therefore each of the cells corresponding to an edge of~$G_1$ must map to itself.  By induction, it follows that every rearrangement of this limit space is the identity.
\end{example}

\subsection{Relation to Thompson's Groups}
\label{subsec:thompson}

In this subsection we show that Thompson's groups $F$, $T$, and $V$ can be realized as rearrangement groups.  We also prove that many different rearrangement groups contain a copy of Thompson's group~$F$.

We assume in this section that the reader is familiar with Thompson's groups.  See \cite{CFP} for a general introduction.

\begin{table}
\centering
\renewcommand{\arraystretch}{1.25}
\newcommand{\mystrut}{\rule{0pt}{0.6in}}
\begin{tabular}{c@{\hspace{0.4in}}c@{\hspace{0.5in}}c}
\hline
Group & Base Graph & Replacement Rule \\
\hline
\raisebox{-0.5ex}{$F$} & \mystrut\vgraphics{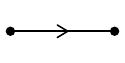} & $\vgraphics{EdgeReplacement1} \;\;\quad\longvarrow\;\;\quad \vgraphics{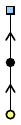}$ \\
\raisebox{-0.5ex}{$T$} & \mystrut\vgraphics{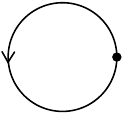} & $\vgraphics{EdgeReplacement1} \;\;\quad\longvarrow\;\;\quad \vgraphics{DoubleEdgeReplacementBlank}$ \\
\raisebox{-0.5ex}{$V$} & \mystrut\vgraphics{OneEdgeBaseGraphBlank} & $\vgraphics{EdgeReplacement1} \;\;\quad\longvarrow\;\;\quad \vgraphics{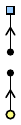}$  \\[0.45in]
\hline
\end{tabular}
\caption{Replacement systems for Thompson's groups $F$, $T$, and $V$.}
\label{tab:ThompsonGroups}
\end{table}
\begin{proposition}\label{prop:ThompsonsGroups}The rearrangement groups corresponding to the replacement systems shown in Table~\ref{tab:ThompsonGroups} are isomorphic to Thompson's groups $F$, $T$, and $V$.
\end{proposition}
\begin{proof}Consider the given replacement system for~$F$.  Let $\mathsf{E}$ denote the edge of the base graph, and let $\mathsf{0}$ and $\mathsf{1}$ denote the edges of the replacement graph.

Each graph $G_n$ in the full expansion sequence for this replacement system is a directed path of length~$2^n$, as shown in Figure~\ref{fig:IntervalFullExpansions}. The gluing relation on the symbol space $\{\mathsf{E}\}\times \{\mathsf{0},\mathsf{1}\}^\infty$ is given by $e\mathsf{0}\overline{\mathsf{1}} \sim e\mathsf{1}\overline{\mathsf{0}}$ for any edge $e$ of~$G_n$, and it follows that the limit space $X$ is homeomorphic with the interval~$[0,1]$, with each point in the symbol space mapping to the point in $[0,1]$ whose binary expansion is the given binary sequence.
\begin{figure}
\centering
$\underset{\textstyle G_1}{\includegraphics{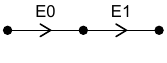}}
\qquad\qquad
\underset{\textstyle G_2}{\includegraphics{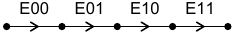}}$
\caption{Two graphs in the full expansion sequence for Thompson's group~$F$.}
\label{fig:IntervalFullExpansions}
\end{figure}%

Under this identification, the gluing vertices for $X$ correspond precisely to the dyadic fractions in $[0,1]$.  The cells in $X$ correspond to the standard dyadic intervals in $[0,1]$, i.e.~all intervals of the form $[(k-1)/2^n,k/2^n]$ for $n\in\mathbb{N}$ and $k\in\{1,\ldots,2^n\}$, and a cellular partition of $X$ is simply any subdivision of $[0,1]$ into standard dyadic intervals.  It is easy to check that the canonical homeomorphism between two standard dyadic intervals is orientation preserving and linear.  Thus, a homeomorphism $h\colon [0,1]\to[0,1]$ is a rearrangement if and only if there exist two partitions $\{I_1,\ldots,I_n\}$ and $\{I_1',\ldots,I_n'\}$ of $[0,1]$ into standard dyadic intervals such that $h$ maps each $I_k$ linearly to $I_k'$  in an orientation-preserving way.  Then the group of rearrangements is precisely Thompson's group~$F$.  Similar arguments hold for $T$ and $V$.
\end{proof}

\begin{table}
\centering
\renewcommand{\arraystretch}{1.25}
\newcommand{\mystrut}{\rule{0pt}{0.7in}}
\begin{tabular}{c@{\hspace{0.4in}}c@{\hspace{0.5in}}c}
\hline
Group & Base Graph & Replacement Rule \\
\hline
\raisebox{-0.5ex}{$F_{3,2}$} & \mystrut\vgraphics{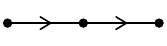} & $\vgraphics{EdgeReplacement1Tall} \;\;\quad\longvarrow\;\;\quad \vgraphics{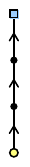}$ \\
\raisebox{-0.5ex}{$T_{3,2}$} & \mystrut\vgraphics{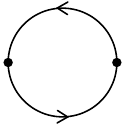} & $\vgraphics{EdgeReplacement1Tall} \;\;\quad\longvarrow\;\;\quad \vgraphics{TripleEdgeReplacement}$ \\
\raisebox{-0.5ex}{$V_{3,2}$} & \mystrut\vgraphics{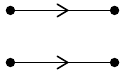} & $\vgraphics{EdgeReplacement1Tall} \;\;\quad\longvarrow\;\;\quad \vgraphics{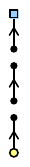}$  \\[0.55in]
\hline
\end{tabular}
\caption{Replacement systems for the generalized Thompson groups $F_{3,2}$, $T_{3,2}$, and $V_{3,2}$.}
\label{tab:GeneralizedThompsonGroups}
\end{table}
\begin{remark} \label{rem:generalizedthompsongroups}
The Thompson groups $F$, $T$, and $V$ belong to the families of \newword{generalized Thompson groups} $F_{n,k}$, $T_{n,k}$, and $V_{n,k}$, where $n$ and $k$ are positive integers. (See \cite{Bro}, where $V_{n,k}$ is denoted~$G_{n,k}$).  The Thompson groups themselves correspond to the case where $n=2$ and $k=1$.  These generalized Thompson groups can also be represented as rearrangement groups, as shown in Table~\ref{tab:GeneralizedThompsonGroups}.

Incidentally, note that the replacement system corresponding to $V_{n,k}$ has trivial gluing relation, so the limit space for $V_{n,k}$ is simply the symbol space $\Omega = \{1,\ldots,k\}\times\{1,\ldots,n\}^\infty$, with any bijection between the cells of two cellular partitions defining a rearrangement.  If $(G_0,e\to R)$ is any replacement system, where $G_0$ has $k$ edges and $R$ has $n$ edges, then the rearrangement group is isomorphic to the subgroup of $V_{n,k}$ consisting of all elements that preserve the corresponding gluing relation.  Since each of the groups $V_{n,k}$ embeds into Thompson's group~$V$, it follows that every rearrangement group embeds into Thompson's group~$V$.
\end{remark}

We now prove that a wide class of rearrangement groups contains copies of Thompson's group~$F$.

\begin{proposition}\label{prop:containsf} Let\/ $(G_0,e\to R)$ be a replacement system with limit space $X$, and suppose the initial vertex of $R$ is a source of degree one, and the terminal vertex of $R$ is a sink of degree one.  Then for any cell $C$ in $X$ the rearrangement group contains a copy of Thompson's group~$F$ that is supported on~$C$.
\end{proposition}
\begin{proof}Let $\iota$ and $\tau$ be the edges of $R$ incident on the initial and terminal vertices, respectively, and let $\{G_n\}$ be the full expansion sequence.  Given any edge $e$ in~$G_n$, let $r_e$ be the rearrangement with graph pair diagram
\[
(G_n\exp e\exp e\iota,\,G_n\exp e\exp e\tau,\,\varphi),
\]
where $\varphi$ is the graph isomorphism defined as follows:
\begin{enumerate}
\item $\varphi$ acts as the identity on edges of $E(G_n) - \{e\}$.
\item $\varphi(e\iota\iota) = e\iota$, $\varphi(e\iota\tau) = e\tau\iota$, and $\varphi(e\tau) = e\tau\tau$.
\item $\varphi(e\iota\zeta) = e\zeta$ and $\varphi(e\zeta) = e\tau\zeta$.
\end{enumerate}
\begin{figure}
\centering
\footnotesize
$
\renewcommand{\arraystretch}{1.25}
\underset{\textstyle\text{Domain\rule{0pt}{10pt}}}{\vgraphics{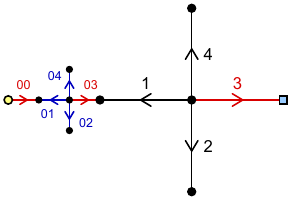}}
\hfill\underset{\textstyle\text{Range\rule{0pt}{10pt}}}{\vgraphics{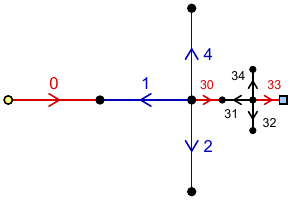}}
\hfill
\scriptstyle\raisebox{-2ex}{$\begin{array}{@{}c@{\;\;\mapsto\;\;}c@{}}
\textsf{00} & \textsf{0} \\
\textsf{01} & \textsf{1} \\
\textsf{02} & \textsf{2} \\
\textsf{03} & \textsf{30} \\
\textsf{04} & \textsf{4} \\
\textsf{1} & \textsf{31} \\
\textsf{2} & \textsf{32} \\
\textsf{3} & \textsf{33} \\
\textsf{4} & \textsf{34}
\end{array}$}$
\caption{The action of $r_e$ on a cell $C(e)$ for the Vicsek fractal, where $\iota = \mathsf{0}$ and $\tau = \mathsf{3}$.  The mappings $e\iota\iota \mapsto e\iota$, $e\iota\tau \mapsto e\tau\iota$, and $e\tau\mapsto e\tau\tau$ are shown in red.}
\label{fig:VicsekFElement}
\end{figure}
For example, Figure~\ref{fig:VicsekFElement} shows the action of $r_e$ on a cell $C(e)$ of the Vicsek fractal, where $\iota = \mathsf{0}$ and $\tau = \mathsf{3}$. We claim that, for any edge $e$ in $G_n$, the rearrangements $r_{e}$ and $r_{e\tau}$ generate a copy of Thompson's group~$F$.

To prove this, recall first that $F$ is given by the presentation
\[
\langle x_0,x_1 \mid x_1x_2=x_3x_1,x_1x_3=x_4x_1\rangle,
\]
where $x_k = x_0^{k-1}x_1x_0^{1-k}$ for $k \geq 2$.  It is easy to check that $r_e$ and $r_{e\tau}$ satisfy these relations.  In particular, $r_e^{k-1}r_{e\tau}r_e^{1-k} = r_{e\tau^k}$ for $k\geq 2$, and these rearrangements satisfy
\[
r_{e\tau} r_{e\tau^2} = r_{e\tau^3} r_{e\tau} \qquad\text{and}\qquad r_{e\tau} r_{e\tau^3} = r_{e\tau^4} r_{e\tau}.
\]
Thus we get a well defined epimorphism $\pi\colon F \to \langle r_e,r_{e\tau}\rangle$.  Since every proper quotient of $F$ is abelian, we can show that $\pi$ is an isomorphism by showing that $r_e$ and $r_{e\tau}$ do not commute.  But $r_er_{e\tau}$ maps the cell $C(e\iota\tau\tau)$ canonically to $C(e\tau\iota\tau)$, and $r_{e\tau}r_e$ maps $C(e\iota\tau\tau)$ canonically to $C(e\tau\tau\iota)$, so these cannot be the same rearrangement.
\end{proof}

Of the examples in Subection~\ref{subsec:examples}, this proposition shows that the basilica Thompson group, all rabbit family rearrangement groups, and all Vicsek family rearrangement groups contain copies of Thompson's group~$F$.  It also follows from this proposition that all of the generalized Thompson groups given in Table~\ref{tab:GeneralizedThompsonGroups} contain a copy of Thompson's group~$F$, though this is well-known.

Note that the Thompson group $\langle r_e,r_{e\tau}\rangle$ defined in the proof of Proposition~\ref{prop:containsf} acts on the Cantor space $\{e\} \times \{\iota,\tau\}^\infty \subset \Omega$ in precisely the same way that Thompson's group $F$ acts on the standard Cantor set $\{0,1\}^\infty$.  In the special case where $\iota$ and $\tau$ share a vertex in~$R$ (e.g.~for the rabbit family), the image of $\{e\} \times \{\iota,\tau\}^\infty$ in the limit space is actually an arc, and the action of $\langle r_e,r_{e\tau}\rangle$ on this arc is conjugate to the action of $F$ on~$[0,1]$.

\subsection{Finite Subgroups}
\label{subsec:finite}

In this subsection, we provide a general characterization of finite subgroups of rearrangement groups.  As an application, we show that the Vicsek rearrangement group is not isomorphic to any generalized Thompson group.

Let $\R = (G_0,e\to R)$ be an expanding replacement system, let $X$ be the corresponding limit space, and let $\G$ be the group of rearrangements of~$X$.  Given an expansion $E$ of~$G_0$, let $\mathrm{Aut}_{\mathcal{R}}(E)$ denote the subgroup of $\G$ consisting of all rearrangements having a graph pair diagrams of the form $(E,E,\varphi)$, where $\varphi$ is an automorphism of~$E$.  Note then that $\mathrm{Aut}_{\R}(E)$ is isomorphic to the group of automorphisms of the directed graph~$E$.

\begin{theorem}\label{thm:FiniteSubgroups}Every finite subgroup of $\G$ is contained in some\/ $\mathrm{Aut}_{\R}(E)$.
\end{theorem}
\begin{proof}Let $H$ be a finite subgroup of $\G$.  For each $h\in H$, let $\mathcal{P}_h$ be a cellular partition of $X$ such that $h$ restricts to a canonical homeomorphism on each cell of $\mathcal{P}_h$.  Let $\mathcal{P}$ be the least common refinement of the partitions $\{\mathcal{P}_h\mid h\in H\}$, i.e.~the set of all minimal cells of the union $\bigcup_{h\in H} \mathcal{P}_h$.  Then each $h\in H$ restricts to a canonical homeomorphism on each cell of~$\mathcal{P}$.

For each $h\in H$, let $h(\mathcal{P})$ denote the image of $\mathcal{P}$ under $h$,
\[
h(\mathcal{P}) = \{h(C) \mid C\in\mathcal{P}\},
\]
and let $\mathcal{P}'$ be the least common refinement of the partitions $\{h(\mathcal{P}) \mid h\in H\}$.  By symmetry, $h(\mathcal{P}') = \mathcal{P}'$ for each $h\in H$.  Moreover, since $\mathcal{P}'$ is a refinement of~$\mathcal{P}$, each $h\in H$ restricts to a canonical homeomorphism on each cell of $\mathcal{P}'$. Then $H$ is a subgroup of $\mathrm{Aut}_{\R}(E)$, where $E$ is the expansion of $G_0$ corresponding to~$\mathcal{P}'$.
\end{proof}

As an application we classify the finite subgroups of the rearrangement group of the Vicsek fractal.

\begin{proposition}\label{prop:FiniteSubgroupsVicsek}Let $\G$ be the rearrangement group for the Vicsek fractal, and let $H$ be a finite group.  Then $H$ is isomorphic to a subgroup of $\G$ if and only if $H$ is solvable of order\/ $2^j 3^k$ for some $j,k\geq 0$.
\end{proposition}
\begin{proof}Note that every expansion $E$ of the base graph for the Vicsek replacement system is a directed tree in which every vertex has degree four or less, and it is easy to show that the automorphism group of such a tree is solvable and has order $2^j 3^k$ for some~$j,k$.

To show that any solvable group of order $2^j 3^k$ is possible, we define a sequence $\{E_0\}_{n=0}^\infty$ of expansions of $G_0$ recursively as follows.  Let $E_0 =G_0$, and for each $n\geq 1$ let $E_n$ be the expansion of $E_{n-1}$ obtained by expanding all of the leaf edges.  Then the automorphism group of $E_n$ is isomorphic to the automorphism group of a complete rooted trinary tree $T_n$ of depth $n+1$ whose root has four children.  But it is easy to see that any solvable group $H$ of order $2^j3^k$ acts faithfully on some~$T_n$.  For example, if
\[
1=H_0 \lhd H_1 \lhd \cdots \lhd H_n = H
\]
is a composition series for $H$, then each $H_{i-1}$ has index 2 or 3 in~$H_i$, so each node in the tree of left cosets of the~$H_i$'s has either two or three children.  The group $H$ acts faithfully on this tree of cosets (since the leaves are the elements of the group), and this can easily be extended to a faithful action of $H$ on~$T_n$.
\end{proof}

\begin{corollary} \label{cor:Vicsek} If $g$ is a rearrangement of the Vicsek fractal of finite order, then\/ $|g| = 2^j3^k$ for some $j,k\in\mathbb{N}$.  Every such order is possible.\hfill\qedsymbol
\end{corollary}

It follows that the rearrangement group of the Vicsek fractal is not isomorphic to $F$, $T$, or $V$, or indeed any other previously known Thompson-like group that the authors are aware of.

Similar methods can be used to show that various rearrangement groups are not isomorphic to one another.  For example, if $n\geq 5$ then the $n$th Vicsek rearrangement group (see Example \ref{ex:VicsekFamily}) and the $(n+1)\text{-st}$ rabbit rearrangement group (see Example \ref{ex:rabbit}) both contain the alternating group $A_n$ but not $A_{n+1}$.  It follows that all of the Vicsek groups are distinct from one another, as are all of the rabbit groups, and none of these groups are isomorphic to a previously known Thompson-like group.  The groups in the Vicsek and rabbit families are also distinct from one another, since each rabbit group has elements of every order, but no group from the Vicsek family has this property.

\subsection{Colored Replacement Systems}
\label{subsec:colors}

In this subsection, we examine the generalization of replacement systems obtained by coloring the edges of the base graph and replacement graphs; in this instance we will allow a different replacement graph for each color.  This generalization allows us to construct rearrangement groups for a wider variety of fractals, and our primary example of such a fractal given in this section is the airplane Julia set, see Example \ref{ex:airplane}.  Additionally, rearrangement groups of colored replacement systems generalize a certain class of diagram group, see Example~\ref{ex:diagramgroups}.

\begin{definition}A \newword{colored replacement system} consists of the following data:
\begin{enumerate}
\item A finite set $C$ of \newword{colors}.
\item A directed base graph $G_0$, whose edges have been colored by the elements of~$C$.
\item For each $c\in C$, a directed replacement graph $R_c$, whose edges have been colored by elements of~$C$.
\end{enumerate}
Each replacement graph $R_c$ has distinguished initial and terminal vertices.
\end{definition}

For a colored replacement system, we always replace a colored edge $e$ with color $c$ by the corresponding replacement graph~$R_c$.

For such a replacement system, the symbol space $\Omega$ can be defined in an obvious way, and it inherits a topology as a closed subspace of the Cantor space $E(G_0) \times \bigl(\bigcup_{c\in C} E(R_c)\bigr)^\infty$.  Assuming the base graph $G_0$ and each of the replacement graphs $R_c$ satisfy the requirements for an expanding replacement system (see Definition~\ref{def:expanding}), the gluing relation $\sim$ (defined as in Definition~\ref{def:LimitSpace}) is an equivalence relation, and the \newword{limit space} $X=\Omega/{\sim}$ is compact and metrizable.

For a colored replacement system $\R$, each cell $C(e)$ in the corresponding limit space has a color, namely the color of the edge $e$, and it only make sense to talk about the canonical homeomorphism between cells of the same color.  With this caveat, rearrangements can be defined as in Definition~\ref{def:Rearrangement}.  Each such rearrangement has a graph pair diagram of the form $(E_1,E_2,\varphi)$, where $E_1$ and $E_2$ are colored expansions of the base graph~$G_0$, and $\varphi\colon E_1\to E_2$ is a color-preserving isomorphism.

For simplicity, in upcoming sections we will focus our theoretical development on monochromatic rearrangement groups.  However, the constructions in Section~\ref{sec:complexes} and analysis of finiteness properties given in Section~\ref{sec:finiteness} carry over almost verbatim to the colored case.

\begin{example}[The Airplane] \label{ex:airplane} Let $C = \{\mathrm{red},\mathrm{blue}\}$, and consider the two replacement rules shown in Figure~\ref{fig:AirplaneReplacementRules}.  Let $\R$ be the replacement system based on these colors and rules, using a single blue edge as the base graph.  The first few stages of the full expansion sequence for $\R$ are shown in Figure~\ref{fig:AirplaneFullExpansions}.
\begin{figure}[p]
\centering
$\vgraphics{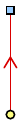} \qquad\longvarrow\qquad \vgraphics{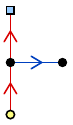}
\qquad\qquad\qquad \vgraphics{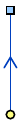} \qquad\longvarrow\qquad \vgraphics{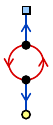}$
\caption{Replacement rules for the airplane replacement system.}
\label{fig:AirplaneReplacementRules}
\end{figure}
\begin{figure}
\centering
$\underset{\textstyle G_1 \rule{0pt}{12pt}}{\vgraphics{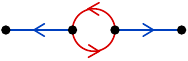}}
\hfill
\underset{\textstyle G_2 \rule{0pt}{12pt}}{\vgraphics{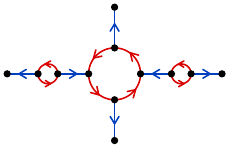}}
\hfill
\underset{\textstyle G_3 \rule{0pt}{12pt}}{\vgraphics{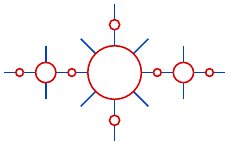}}$
\caption{Three graphs from the full expansion sequence for the airplane replacement system.  For clarity, we have not drawn vertices or arrows on~$G_3$.}
\label{fig:AirplaneFullExpansions}
\end{figure}
\begin{figure}
\centering
\includegraphics[trim={0 3pt 0 3pt},clip]{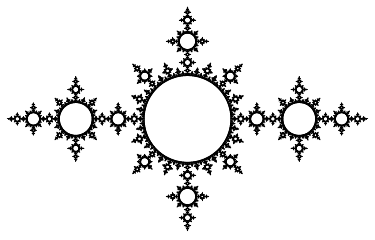}
\vspace{-1ex}
\caption{The limit space for the airplane replacement system.  This fractal is homeomorphic to the airplane Julia set.}
\label{fig:AirplaneFractal}
\end{figure}

Figure~\ref{fig:AirplaneFractal} shows the resulting limit space.  This fractal is homeomorphic to the Julia set for $z^2 - 1.755$, which is known as the \newword{airplane}. As with the basilica and the rabbits, there is an entire interior component of the Mandelbrot set whose corresponding Julia sets have the structure of the airplane, as shown in Figure~\ref{fig:AirplaneBulb}.
\begin{figure}
\centering
\includegraphics{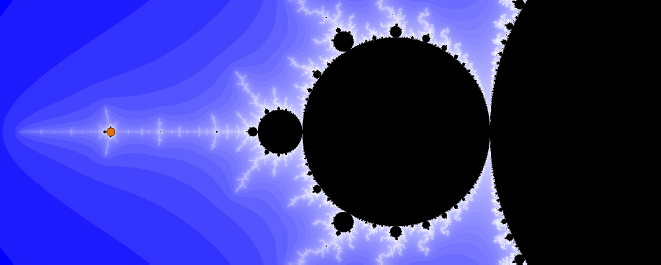}
\vspace{-1ex}\caption{The interior component of the Mandelbrot set that corresponds to the airplane, where the black region on the very right is the main cardioid.\label{fig:AirplaneBulb}}
\end{figure}

Though we will not prove it here, the group of rearrangements for the airplane has type~$F_\infty$.
\end{example}

\begin{example}[Diagram Groups] \label{ex:diagramgroups} A colored replacement system is \newword{linear} if the base graph $G_0$ is a directed path, and the replacement graph $R_c$ for each color is a directed path of length two or greater from the initial vertex to the terminal vertex.  Such a replacement system always has a limit space homeomorphic to a closed interval.

\begin{figure}
\centering
$\vgraphics{ColoredReplacementRed} \qquad\longvarrow\qquad \vgraphics{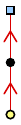}
\qquad\qquad\qquad \vgraphics{ColoredReplacementBlue} \qquad\longvarrow\qquad \vgraphics{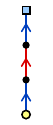}$
\caption{Replacement rules for a linear colored replacement system.}
\label{fig:LinearReplacementRules}
\end{figure}
For example, consider the pair of replacement rules shown in Figure~\ref{fig:LinearReplacementRules}, involving two colors red and blue. Let $\R$ be the replacement system based on these rules, having a single blue edge as its base graph.  Then the corresponding rearrangement group $\G$ acts on a closed interval.  Algebraically, $\G$ is isomorphic to the restricted wreath product~$F\wr F$, where the wreath product is defined using the action of Thompson's group $F$ on the dyadic points in~$(0,1)$.  It is not hard to show that this group has type~$F_\infty$.

Guba and Sapir defined the class of \newword{diagram groups} associated to semigroup presentations~\cite{GuSa}.  The rearrangement group corresponding to a linear colored replacement system is always isomorphic to a diagram group, where the replacement rules determine the presentation of the corresponding semigroup. For example, the rearrangement group $\G$ describe above is isomorphic to a diagram group over the semigroup presentation $\langle R,B \mid R=R^2, B=BRB\rangle$. The $\CAT(0)$ complex that we construct for $\G$ in Section~\ref{sec:complexes} is the same as the complex constructed by Farley in \cite{Farley1} for this diagram group.
\end{example}

\section{Cubical Complexes}
\label{sec:complexes}

In this section we define certain $\mathrm{CAT}(0)$ complexes associated to rearrangement groups.  We will assume the reader is familiar with the language of $\mathrm{CAT}(0)$ cubical complexes---see \cite{BriHae} for a comprehensive introduction to this subject.

In \cite{Farley1} and \cite{Farley2}, Farley constructed a locally finite $\mathrm{CAT}(0)$ cubical complex associated to each of the Thompson groups $F$, $T$, and $V$, as well as the diagram groups of Guba and Sapir~\cite{GuSa}, and some of their generalizations. Our construction of a $\CAT(0)$ complex for rearrangement groups is very similar to Farley's, and our complex is the same as Farley's in the cases of $F$, $T$, and~$V$.

Before defining the complex, we need to expand our definition of rearrangement to include homeomorphisms between certain pairs of limit spaces.  This generalization gives rearrangements a groupoid structure and is discussed in Subsection~\ref{subsec:genrearrange}.  We use these generalized rearrangements to define the 1-skeleton of the cubical complex in Subsection~\ref{subsec:complex}.  The technical background needed to define the cubes of the complex is given in Subsections \ref{subsec:contraction} and \ref{subsec:partial order}.  Subsection~\ref{subsec:cubes} defines the cubes in our complex and proves that the complex is $\CAT(0)$.

\subsection{A Groupoid of Rearrangements}
\label{subsec:genrearrange}

We now introduce the notion of a rearrangement between two limit spaces.  We will use this idea heavily in the construction of $\CAT(0)$ cubical complexes.

If we fix a replacement rule $e\to R$, any graph $G$ can serve as the base graph for a replacement system $(G, e\to R)$, leading to a limit space~$X(G)$.  For example, Figure~\ref{fig:DifferentBasilicaBase} shows a limit space obtained from the basilica replacement rule using a different base graph.  Note that this space is actually homemorphic to the basilica, but its structure as a limit space is different (see~Remark~\ref{rem:structure}).
\begin{figure}[tb]
\centering
\vgraphics{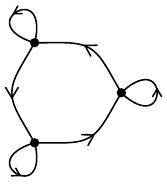}
\qquad\qquad\vgraphics{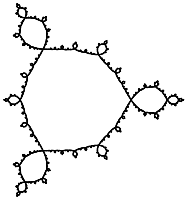}
\caption{A base graph $G$ and the corresponding limit space $X(G)$ for the basilica replacement rule.}
\label{fig:DifferentBasilicaBase}
\end{figure}

It is possible to define canonical homeomorphisms between the cells of any two limit spaces that are defined using the same replacement rule.  This leads to the following definition.

\begin{definition}Let $X(G)$ and $X(G')$ be limit spaces based on the same replacement rule.  A homeomorphism $f\colon X(G)\to X(G')$ is called a \newword{rearrangement} if there exists a cellular partition $\mathcal{P}$ of $X(G)$ such that $f$ restricts to a canonical homeomorphism on each cell of~$\mathcal{P}$.
\end{definition}

For example, Figure~\ref{Fig:DifferentBasilicasRearrangement} shows a rearrangement between two different limit spaces, both of which are based on the basilica replacement rule.
\begin{figure}[b]
\centering
\vgraphics{BasilicaBetaDomain}\quad\longvarrow\quad\vgraphics{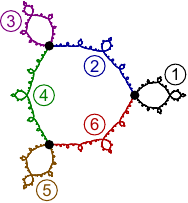}
\caption{A rearrangement from the basilica Julia set to another limit space based on the same replacement rule.  Each of the numbered cells on the left maps to the corresponding cell on the right via a canonical homeomorphism.}
\label{Fig:DifferentBasilicasRearrangement}
\end{figure}

The class of all rearrangements between limit spaces corresponding to a single replacement rule $e\to R$ forms a category with inverses (i.e.~a groupoid) under composition.  The objects of this category consist of one limit space $X(G)$ for each finite directed graph~$G$ (up to isomorphism), and the morphisms are the rearrangements between them.

Any rearrangement $X(G)\to X(G')$ can be described by a graph pair diagram $(E,E',\varphi)$, where $E$ and $E'$ are expansions of $G$ and $G'$, respectively, and $\varphi\colon E\to E'$ is an isomorphism.  Indeed, every such rearrangement has a unique reduced graph pair diagram.

Note that there may or may not exist a rearrangement between a given pair of limit spaces $X(G)$ and $X(G')$.  Specifically, such a rearrangement exists if and only if $G$ and $G'$ have at least one isomorphic pair of expansions.

For the remainder of this section, we let $\R = (G_0,e\to R)$ be an expanding replacement system.

\begin{definition}
The \newword{graph family} $\Gamma(\mathcal{R})$ is the set of all finite, directed graphs $G$ for which there exists at least one rearrangement $X(G_0) \to X(G)$.
\end{definition}

That is, the graph family for $\R$ is the set of all finite directed graphs $G$ that have at least one expansion isomorphic to an expansion of~$G_0$.  In particular, every graph $G$ in $\Gamma(\R)$ has an expansion isomorphic to some graph in the full expansion sequence for~$\R$.

From an algebraic point of view, the graph family $\Gamma(\mathcal{R})$ is precisely the set of graphs whose corresponding limit spaces lie in the connected component of $X(G_0)$ in the groupoid of rearrangements.  If $G \in \Gamma(\mathcal{R})$, it follows that the rearrangement groups for the limit spaces $X(G_0)$ and $X(G)$ are isomorphic, with any rearrangement $X(G_0)\to X(G)$ conjugating one to the other.

\begin{example}\label{ex:VicsekGraphFamily} Let $\R_n$ be a replacement system from the Vicsek family shown in Table~\ref{tab:VicsekFamily}.  Then $\Gamma(\R_n)$ consists of all finite, directed graphs $G$ that satisfy the following conditions:
\begin{enumerate}
\item $G$ is a tree, all of whose vertices are either sources or sinks.
\item Each source in $G$ has exactly $n$ outgoing edges.
\item Each sink in $G$ has either one or two incoming edges.
\end{enumerate}
Indeed, every graph satisfying these conditions is an expansion of the base graph for~$\R_n$.
\end{example}

\begin{example}\label{ex:RabbitGraphFamily} Let $\R_n$ be a replacement system from the rabbit family shown in Table~\ref{tab:Rabbits} (e.g.~$\R_1$ is the basilica replacement system).  Then $\Gamma(\R_n)$ consists of all finite, connected, directed graphs $G$ that satisfy the following conditions:
\begin{enumerate}
\item Each vertex of $G$ has $n+1$ incoming edges and $n+1$ outgoing edges.
\item Removing any vertex of $G$ cuts the graph into $n+1$ connected components.
\end{enumerate}
Indeed, every graph satisfying these conditions is an expansion of the base graph for~$\R_n$.
\end{example}

\begin{example}\label{ex:BubbleBathGraphFamily} Let $\R$ be the replacement system for the Julia set of a rational map shown in Figure~\ref{fig:BubbleBathAll}. Then $\Gamma(\R)$ consists of all finite, connected graphs $G$ that satisfy the following conditions:
\begin{enumerate}
\item Every vertex of $G$ has degree $3$.
\item $G$ is a series-parallel graph, i.e.~it has no subgraph homeomorphic to the complete graph~$K_4$.
\end{enumerate}
Since the replacement graph for $\R$ is symmetric between the initial and terminal vertices, it is possible to switch the direction of an edge by expanding and then contracting, and hence there is no restriction on the directions of the edges for graphs in~$\Gamma(\R)$.  As a result, $\Gamma(\R)$~contains many graphs that are not isomorphic (as directed graphs) to any expansion of the base graph.
\end{example}

\subsection{The Complex}
\label{subsec:complex}

In this section we define a directed graph $K^1(\G)$ on which the group $\G$ acts properly by automorphisms, where $\G$ is the  group of rearrangements associated to $\mathcal{R} = (G_0,e\to R)$.   We will prove in Section~\ref{subsec:cubes} that this graph is the $1$-skeleton of a $\CAT(0)$ cubical complex.

The definition of $K^1(\G)$ is based on a certain special rearrangements that generate the full groupoid, namely base isomorphisms, simple expansions morphisms, and simple contraction morphisms.  We begin with base isomorphisms.

\begin{definition}\quad
Given an isomorphism $\varphi \colon G_1\to G_2$ of directed graphs, the corresponding \newword{base isomorphism} $\varphi\colon X(G_1)\to X(G_2)$ is the rearrangement whose graph pair diagram is $(G_1,G_2,\varphi)$.  A base isomorphism $\varphi \colon X(G) \to X(G)$ is a \newword{base automorphism}.
\end{definition}

We will abuse notation by using the same letter $\varphi$ to refer to an isomorphism $\varphi\colon G_1\to G_2$ of directed graphs and the corresponding base isomorphism $\varphi\colon X(G_1)\to X(G_2)$.

\begin{definition}
An \newword{expansion morphism} is a rearrangement $x\colon X(G)\to X(G')$ whose reduced graph pair diagram has the form $(E,G',\varphi)$ for some expansion $E$ of $G$.  If $E$ is a simple expansion of~$G$, then $x$ is a \newword{simple expansion morphism}.
\end{definition}

The inverse of an expansion morphism is called a \newword{contraction morphism}, and the inverse of a simple expansion morphism is a \newword{simple contraction morphism}.  Note that base isomorphisms are both expansions and contractions, and every expansion that is not a base isomorphism is a composition of simple expansions.

Given a graph $G$ and an expansion $E$ of $G$, the corresponding \newword{canonical expansion morphism} $x\colon X(G) \to X(E)$ is the rearrangement with graph pair diagram $(E,E,\mathrm{id})$. If $f\colon X(G) \to X(G')$ is any rearrangement with graph pair diagram $(E,E',\varphi)$, then $f$ can be written as a composition $(x')^{-1} \circ \varphi \circ x$, where $x\colon X(G)\to X(E)$ and $x'\colon X(G')\to X(E')$ are the canonical expansion morphisms.

\begin{definition}Two rearrangements $f\colon X(G)\to X(G_1)$ and $g\colon X(G)\to X(G_2)$ are \newword{range equivalent} if there exists a base isomorphism $\varphi\colon X(G_1)\to X(G_2)$ such that $g=\varphi\circ f$.
\end{definition}

Note that range equivalent rearrangements always have the same domain, but may have different codomains.  As the name implies, range equivalence is an equivalence relation on rearrangements with a given domain.  The range equivalence class of a rearrangement $f$ will be denoted~$[f]$.

It is easy to see that two expansion morphisms with the same domain are range equivalent if and only if their reduced graph pair diagrams have the same domain expansion.  In particular, every expansion morphism is range equivalent to a unique canonical expansion morphism.  The criteria for range equivalence of contractions is more subtle, and will be explored in Section~\ref{subsec:contraction}.

We are now ready to define the 1-skeleton of our complex.

\begin{definition}Let $K^1(\G)$ be the directed graph defined as follows:
\begin{enumerate}
\item The vertices of $K^1(\G)$ are the range equivalence classes of rearrangements with domain~$X(G_0)$.
\item There is an directed edge from $[f]$ to $[g]$ if $g = x\circ f$ for some simple expansion morphism~$x$.
\end{enumerate}
\end{definition}

Note that the definition of the directed edges is independent of the chosen representatives $f$ and $g$.  In particular, if $\varphi
\circ f$ and $\psi \circ g$ are another pair of representatives for $[f]$ and $[g]$, then $\psi\circ g = (\psi\circ x \circ \varphi^{-1})\circ (\varphi\circ f)$, where $\psi\circ x \circ \varphi^{-1}$ is again a simple expansion morphism.

We will let $K^0(\G)$ denote the set of vertices of the graph $K^1(\G)$.  The group $\G$ acts on $K^0(\G)$ by right composition, i.e.~$[f]g = [f\circ g]$ for any rearrangement $f\colon X(G_0) \to X(G)$ and any $g\in \G$.  It is easy to see that this extends to an action of $\G$ on $K^1(\G)$ by automorphisms.

\begin{proposition}The action of\/ $\G$ on $K^1(\G)$ is proper.  In particular, given any vertex\/ $[f] \in K^0(\G)$, where $f\colon X(G_0) \to X(G)$, the stabilizer of\/ $[f]$ is isomorphic to the group of automorphisms of the graph~$G$.
\end{proposition}
\begin{proof}If $g\in \G$ then $[f]g = [f]$ if and only if $f\circ g$ is range equivalent to $f$, i.e.~if and only if $f\circ g=\varphi\circ f$ for some base automorphism $\varphi$ of $X(G)$.  Thus the stabilizer of $[f]$ consists of all rearrangements of the form $f^{-1}\circ \varphi\circ f$, where $\varphi$ is a base automorphism of~$X(G)$. This is isomorphic to to the group of base automorphisms of $X(G)$, which is itself naturally isomorphic to the automorphism group of~$G$.
\end{proof}

We will prove in Section~\ref{subsec:cubes} that $K^1(\G)$ is the 1-skeleton of a $\CAT(0)$ cubical complex~$K(\G)$.  It follows immediately that the action of $\G$ on $K^1(\G)$ extends to a proper action of $\G$ on $K(\G)$ by isometries.

\subsection{Contractions}
\label{subsec:contraction}

The goal of this section is to enumerate the incoming edges at a vertex $[f]$ of $K^1(\G)$.  Such edges correspond to simple contraction morphisms.  To enumerate them, we need to describe all possible ways of contracting a given graph~$G$.

For the following definition, let $R$ denote the replacement graph, and let $\Rloop$ be the graph obtained from $R$ by identifying the initial and terminal vertices.

\begin{definition}Let $G$ be a directed graph.  A \newword{characteristic map} for $R$ in $G$ is an isomorphism $\chi\colon R\to S$ or $\chi\colon \Rloop\to S$, where $S$ is a subgraph of $G$, having the following property: for each interior vertex $v$ of~$R$, every edge of $G$ incident on $\chi(v)$ lies in~$S$.
\end{definition}

A subgraph $S$ of $G$ is called a \newword{collapsible subgraph} if it is the image of some characteristic map.  Note that a single collapsible subgraph $S$ may be the image of more than one characteristic map when $R$ or $\Rloop$ has nontrivial automorphisms.

\begin{figure}
\centering
\vgraphics{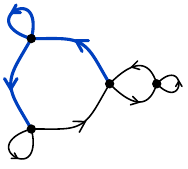} \qquad \vgraphics{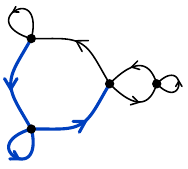} \qquad \vgraphics{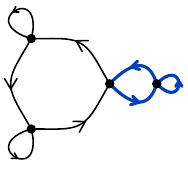}
\caption{Three collapsible subgraphs of a graph in the basilica graph family.}
\label{fig:BasilicaCollapsibleSubgraphs}
\end{figure}
\begin{example} Figure~\ref{fig:BasilicaCollapsibleSubgraphs} shows the three collapsible subgraphs for a certain graph $G$ that lies in the graph family for the basilica (see  Example~\ref{ex:RabbitGraphFamily}).  The two on the left are each isomorphic to~$R$, while the one on the right is isomorphic to~$\Rloop$. Each of these collapsible subgraphs corresponds to a uniquely determined characteristic map.
\end{example}

\begin{figure}
\centering
\begin{tabular}{c@{\qquad}c}
\includegraphics{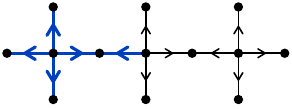} & \includegraphics{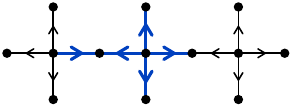} \\[12pt]
\includegraphics{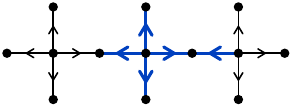} & \includegraphics{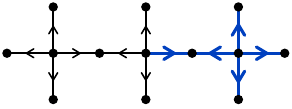}
\end{tabular}
\caption{Four collapsible subgraphs of a graph in the Vicsek graph family, corresponding to a total of sixteen different characteristic maps.}
\label{fig:VicsekCollapsibleSubgraphs}
\end{figure}
\begin{example}Figure~\ref{fig:VicsekCollapsibleSubgraphs} shows the four collapsible subgraphs for a certain graph $G$ that lies in the graph family for the Vicsek replacement system (see Example~\ref{ex:VicsekFamily}).  Of these four collapsible subgraphs, the first and last each correspond to six different characteristic maps, corresponding to the six permutations of the dangling edges, while the collapsible subgraphs in the center each correspond to two different characteristic maps.
\end{example}

\begin{definition}Let $c\colon X(G) \to X(G')$ be a simple contraction with graph pair diagram $(G,E',\varphi)$, where $E'$ is the expansion of $G'$ obtained by replacing a single edge $e_0$ and $\varphi\colon G\to E'$ is an isomorphism. The \newword{characteristic} of $c$ is the isomorphism $\chi\colon R\to S$ or $\chi\colon \Rloop\to S$ that maps each edge $\edge$ of $R$ to the edge $\varphi^{-1}(e_0\edge)$ of~$S$.
\end{definition}

It is easy to check that this function $\chi$ is a characteristic map.

\begin{proposition}Let $G$ be a directed graph.  Then every characteristic map $\chi$ for $R$ in $G$ is the characteristic of some simple contraction with domain~$X(G)$, and two simple contractions with domain $X(G)$ are range equivalent if and only if they have the same characteristic.
\end{proposition}
\begin{proof}First let $\chi$ be a characteristic map for $R$ in $G$, with image~$S$.  Let $G'$ be the graph obtained by removing $S$ and replacing it with a single directed edge~$e_0$, oriented in the appropriate direction, and let $E'$ be the expansion of $G'$ obtained by replacing the edge~$e_0$.  Define an isomorphism $\varphi\colon G \to E'$ by letting $\varphi(e) = e$ for edges $e$ that do not lie in~$S$, and $\varphi(\chi(\edge)) = e_0\edge$ for every edge $\chi(\edge)$ of $S$.  Then $(G,E',\varphi)$ is the graph pair diagram for a simple contraction $X(G) \to X(G')$ having $\chi$ as its characteristic.

Now suppose that $x_i\colon X(G) \to X(G_i)$ (for $i=1,2$) are two simple contractions with graph pair diagrams $(G,E_i,\varphi_i)$, and let $e_i$ be the edge of $G_i$ that was replaced to obtain $E_i$. Observe that $(E_1,E_2,\varphi_2\circ \varphi_1^{-1})$ is a (possibly unreduced) graph pair diagram for $x_2\circ x_1^{-1}$.  Then $x_2\circ x_1^{-1}$ is a base isomorphism if and only if this graph pair diagram can be reduced, i.e.~if and only if $(\varphi_2\circ\varphi_1^{-1})(e_1\edge) = e_2\edge$ for every edge $\edge$ of~$R$.  This occurs if and only if $\varphi_1^{-1}(e_1\edge) = \varphi_2^{-1}(e_2\edge)$ for every edge $\edge$ of~$R$, which is to say that $x_1$ and $x_2$ have the same characteristic.
\end{proof}

Since the number of possible characteristic maps $\chi\colon R\to G$ is finite, it follows that the graph $K^1(\G)$ is locally finite.

\subsection{A Partial Order}
\label{subsec:partial order}

In this section, we define a partial order on $K^0(\G)$, and we prove that every finite subset of $K^0(\G)$ has a least upper bound.  We will use this partial order in Section~\ref{subsec:cubes} to prove that the complex $K(\G)$ is~$\CAT(0)$.

Given any~$G\in\Gamma(\R)$, let $\rset{G}$ denote the set of all range equivalence classes of rearrangements having domain~$X(G)$.  For example, $\rset{G_0}$ is equal to the set of vertices~$K^0(\G)$.

\begin{definition} Given rearrangements $f,g\in\rset{G}$, we say that $[f]$ \newword{precedes} $[g]$, denoted $[f] \preceq [g]$, if $g = x \circ f$ for some expansion morphism~$x$.
\end{definition}

It is easy to show that, for any $G\in\Gamma(\R)$, the relation $\preceq$ is a well-defined partial order on $\rset{G}$.  In addition, if $f\colon X(G) \to X(G')$ is a rearrangement, then right-composition by $f$ induces an order isomorphism from $\rset{G'}$ to $\rset{G}$.

Note that the graph $K^1(\G)$ is precisely the Hasse diagram for $K^0(\G)$ under the partial order~$\preceq$.  In particular, since every expansion that is not a base isomorphism is a composition of simple expansions, two vertices $[f]$ and $[g]$ are joined by a directed path in $K^1(\G)$ if and only if $[f]\preceq[g]$.

Given a vertex $[f] \in K^0(\G)$, where $f\colon X(G_0) \to X(G)$, the \newword{rank} of $[f]$ is the number of edges in the graph~$G$.  Since isomorphic graphs have the same number of edges, the rank of $[f]$ does not depend on the chosen representative~$f$.  If $[f] \preceq [g]$, it follows that the rank of $[f]$ is less than or equal to the rank of $[g]$, with equality if and only if $[f]=[g]$.  Note that $K^0(\G)$ is not quite a ranked poset with respect to this definition, since directed edges typically increase the rank by more than one.

\begin{lemma} \label{lem:unique} Let $f\colon X(G) \to X(G')$ be a rearrangement, and let\/ $\mathrm{id}$ be the identity map on~$X(G)$.  Then\/ $[f]$ and\/ $[\mathrm{id}]$ have a least upper bound in~$\rset{G}$.
\end{lemma}
\begin{proof}This is just a restatement of the fact that $f$ has a unique reduced graph pair diagram.  In particular, if $[x]\in\rset{G}$, observe that $[\mathrm{id}] \preceq [x]$ if and only if $x$ is an expansion morphism.  Indeed, we need only consider the case where $x$ is canonical, corresponding to some expansion $E$ of~$G$. Then $[f] \preceq [x]$ if and only if $x\circ f^{-1}$ is an expansion, which occurs if and only if $f$ has a graph pair diagram with $E$ as the domain graph.  Then the minimum such $x$ is the canonical expansion morphism corresponding to the domain graph of the reduced graph pair diagram for~$f$.
\end{proof}

\begin{proposition}\label{prop:lub}Let $G \in \Gamma(\R)$, and let $f,g\in \rset{G}$.  Then\/ $[f]$ and\/ $[g]$ have a least upper bound in~$\rset{G}$.
\end{proposition}
\begin{proof}Let $X(G')$ denote the codomain of $f$, and let $\mathrm{id}$ be the identity map on $X(G')$.  By Lemma~\ref{lem:unique}, there exists a least upper bound $[x]$ for $[\mathrm{id}]$ and $[g\circ f^{-1}]$ in $\rset{G'}$.  Then $[x\circ f]$ is the least upper bound for $[f]$ and $[g]$  in $\rset{G}$, since right-composition by $f$ is an order isomorphism $\rset{G'} \to \rset{G}$.
\end{proof}

In particular, every pair of elements of $K^0(\G)$ has a least upper bound.  By induction, any finite subset of $K^0(\G)$ has a least upper bound as well.

\subsection{Cubes}
\label{subsec:cubes}

In this section we define the complex $K(\G)$ and prove that it is $\CAT(0)$.  Our proof closely follows the methods of Farley~\cite{Farley1,Farley2}.

Given any directed graph $G$ and any set $S$ of edges of $G$, let $G \exp S$ denote the expansion of $G$ obtained by replacing each of the edges of~$S$. That is, if $S=\{e_1,\ldots,e_n\}$, then
\[
G \exp S \;=\; G\exp e_1\exp \cdots \exp e_n.
\]
Let $x_S \colon X(G) \to X(G \exp S)$ denote the canonical expansion with graph pair diagram $(G \exp S,G \exp S,\mathrm{id})$.

\begin{definition}If $f\colon X(G_0) \to X(G)$ is a rearrangement and $S$ is a set of edges of $G$, the corresponding \newword{cube} is the subset of $K^0(\G)$ defined by
\[
\cube(f,S) \;=\; \{[x_T\circ f] \mid T\subseteq S\}.
\]
\end{definition}

\begin{proposition}The subgraph of $K^1(\G)$ induced by\/ $\cube(f,S)$ is isomorphic to the\/ $1$-skeleton of an\/ $|S|$-dimensional cube.
\end{proposition}
\begin{proof}Note first that $\cube(f,S)$ has $2^{|S|}$ distinct vertices, since $[x_T\circ f] = [x_U \circ f]$ if and only if $[x_T] = [x_U]$, which occurs if and only if $T=U$.  Moreover, observe that $[x_T\circ f] \preceq [x_U\circ f]$ if and only if $T\subseteq U$, with $x_U\circ f = x_{U-T}\circ (x_T\circ f)$.  But $x_{U-T}$ is a simple expansion if and only if $|U-T| = 1$.  We conclude that there is an edge in $K^1(\G)$ from $[x_T\circ f]$ to $[x_U\circ f]$ if and only if $T\subseteq U$ and $|U-T|=1$, and therefore the induced subgraph is the $1$-skeleton of a cube.
\end{proof}

It is easy to check that the faces of $\cube(f,S)$ are the sets
\[
\bigl\{ \cube(x_T\circ f,U) \;\bigr|\; T,U\subseteq S\text{ and }T\cap U=\emptyset\bigr\}.
\]
The vertex $[f]$ is the minimum element of $\cube(f,S)$ under the partial order~$\preceq$, and is called the \newword{base} of the cube.  The vertex $[x_S\circ f]$ is the \newword{apex} of the cube, and is the maximum element under the partial order.  Indeed, $\cube(f,S)$ itself is precisely the closed interval $\bigl[[f],[x_S\circ f]\bigr]$, and the faces of the cube are precisely the closed subintervals of this interval.


Note also that $\cube(f,S) = \cube\bigl(\varphi\circ f,\varphi(S)\bigr)$ for any base isomorphism~$\varphi$.  In particular, we obtain the same cubes based at $[f]$ no matter what representative we choose for~$[f]$.

\begin{proposition}The intersection of two cubes is either empty or is a common face of each.
\end{proposition}
\begin{proof}Let $[b_1,a_1]$ and $[b_2,a_2]$ be two cubes with bases $b_1,b_2$ and apexes $a_1,a_2$, and suppose that the cubes intersect at some vertex~$v$. By Proposition~\ref{prop:lub}, the vertices $b_1$ and $b_2$ have a least upper bound~$b_3$.  Then $b_3 \in [b_1,v] \subseteq [b_1,a_1]$ and similarly $b_3 \in [b_2,a_2]$, so $b_3$ lies in both cubes. Let $a_3$ be the least upper bound of all the vertices in the intersection of the two cubes.  Then again $a_3$ must lie in both cubes, and $[b_1,a_1]\cap [b_2,a_2]$ is precisely the cube $[b_3,a_3]$, which is a face of each.
\end{proof}

We conclude that the sets $\cube(f,S)$ form an abstract cubical complex (see~\cite{Farley1}).  Let $K(\G)$ be the geometric realization of this abstract complex, and note that the $1$-skeleton of $K(\G)$ is indeed $K^1(\G)$, since the directed edges $\bigl([f],[x_{\{e\}}\circ f]\bigr)$ of $K^1(\G)$ are precisely the $1$-cubes of the form $\cube(f,\{e\})$.

\begin{proposition}The complex $K(\G)$ is contractible.
\label{prop:contractible}
\end{proposition}
\begin{proof}We will follow the nerve cover argument for the Farley complex, which the authors learned from K.~Bux.  For $v\in K^0(\G)$, let $U(v)$ denote the subcomplex of $K(\G)$ induced by the set of all vertices $w\in K^0(\G)$ for which $v\preceq w$.  We claim that each $U(v)$ is contractible.

To see this, consider the filtration $\{U_k(v)\}$ of $U(v)$, where $U_k(v)$ is the subcomplex induced by all vertices of rank~$\leq k$.  For $k > \mathrm{rank}(v)$, note that each vertex of rank $k$ in $U_k(v)$ is the apex of a maximal cube.  Then $U_k(v)$ can be collapsed onto $U_{k-1}(v)$, since each vertex of rank $k$ is a free face.  It follows that each $U_k(v)$ is contractible, and hence $U(v)$ is contractible.

Now, if $v_1,\ldots, v_n \in K^0(\G)$, then $U(v_1) \cap \cdots \cap U(v_n) = U(w)$, where $w$ is the least upper bound of $v_1,\ldots,v_n$.  In particular, $U(v_1)\cap\cdots\cap U(v_n)$ is nonempty and contractible for all $v_1,\ldots,v_n$.  Then $K(\G)$ is homotopy equivalent to the nerve of the cover $\{U(v) \mid v\in K^0(\G)\}$, which is an infinite-dimensional simplex.
\end{proof}

For the following theorem, define the \newword{support} of a simple expansion with domain $X(G)$ to be the edge of $G$ that it expands, and the support of a simple contraction with domain $X(G)$ to be the set of edges of the corresponding collapsible subgraph of~$G$.

\begin{theorem}\label{thm:Links}Let $[f]$ be a vertex of $K(\G)$, and let\/ $[x_1\circ f],\ldots,[x_n\circ f]$ be distinct vertices adjacent to\/ $[f]$, where each $x_i$ is either a simple contraction or a simple expansion.  Then the vertices\/ $[f],[x_1\circ f],\ldots,[x_n\circ f]$ lie in a common cube of $K(\G)$ if and only if the supports of\/ $x_1,\ldots,x_n$ are pairwise disjoint.
\end{theorem}
\begin{proof}Let $f\colon X(G_0) \to X(G)$, and suppose first that the given vertices lie in some $\cube(g,S)$.  Then $[f] = [x_T\circ g]$ for some $T\subseteq S$, and in particular $f = \varphi\circ (x_T\circ g)$ for some base isomorphism $\varphi$.  We also know that $[x_i\circ f] = [x_{T_i}\circ g]$ for each $i$, where $T_i \subseteq S$ is obtained from $T$ by adding or removing a single edge.  If $T_i = T \cup \{e\}$, then the support of $x_i$ is $\{\varphi(e)\}$.  If $T_i = T - \{e\}$, then the support of $x_i$ is $\{\varphi(e\edge) \mid \edge \in E(R)\}$.  These sets are clearly disjoint.

Now suppose that the supports of $x_1,\ldots,x_n$ are pairwise disjoint.  We can assume that $x_1,\ldots,x_m$ are simple contractions with characteristics $\chi_1,\ldots,\chi_m$, and $x_{m+1},\ldots,x_n$ are simple expansions with supports $\{e_{m+1}\},\ldots,\{e_n\}$.  Let $G'$ be the graph obtained from $G$ by replacing the image of each $\chi_i$ by an edge $e_i$ with the appropriate orientation, and let $T = \{e_1,\ldots,e_m\}$ and $S=\{e_1,\ldots,e_n\}$.  Then $G' \exp T$ is canonically isomorphic to~$G$, via the isomorphism $\varphi$ that acts as the identity on $E(G')-T$, and maps $e_i\edge$ to $\chi_i(\edge)$ for each $i\leq m$ and each edge $\edge$ of~$R$.

Let $g\colon X(G_0) \to X(G')$ be the rearrangement $x_T^{-1} \circ \varphi^{-1} \circ f$.  We claim that all of the desired vertices lie in $\cube(g,S)$.   Note first that $[f] = [x_T\circ g]$.  Next, for $i\leq m$, observe that the rearrangements $x_{T-\{e_i\}} \circ x_T^{-1} \circ \varphi^{-1}$ and $x_i$ with domain $X(G)$ are range equivalent.  Hence $[x_i\circ f] = [x_{T-\{e_i\}} \circ x_T^{-1} \circ \varphi^{-1} \circ f] =  [x_{T-\{e_i\}}\circ g]$.  If $i > m$, the rearrangements $x_{T\cup \{e_i\}} \circ x_T^{-1} \circ \varphi^{-1}$ and $x_i$ with domain $X(G)$ are range equivalent, and so $[x_i\circ f] = [x_{T\cup \{e_i\}} \circ x_T^{-1} \circ \varphi^{-1}\circ f] = [x_{T\cup \{e_i\}}\circ g]$.
\end{proof}

\begin{corollary}The complex $K(\G)$ is\/ $\CAT(0)$.
\label{cor:cat0}
\end{corollary}
\begin{proof}We have shown that $K(\G)$ is contractible, and it follows from the previous theorem that the link of every vertex in $K(\G)$ is a flag complex.  Then $K(\G)$ is $\CAT(0)$ by Gromov's theorem (see Theorem II.5.20 in~\cite{BriHae}).
\end{proof}

\begin{remark} As mentioned previously, in the case where $\G$ is one of the three Thompson groups, the complex $K(\G)$ is precisely the associated Farley complex.  More generally, if the replacement graph $R$ has $n$ edges and the base graph $G_0$ has $k$ edges, then the complex $K(\G)$ embeds isometrically into the Farley complex for the generalized Thompson group~$V_{n,k}$ (see Remark~\ref{rem:generalizedthompsongroups}).
\end{remark}

\section{Finiteness Properties}
\label{sec:finiteness}

A group is of \newword{type $\boldsymbol{F_n}$} if it is the fundamental group of an aspherical CW complex with finite $n$-skeleton.  For example, a group is of type~$F_1$ if and only it is finitely generated, and a group is of type~$F_2$ if and only if it is finitely presented.  A group is of \newword{type $\boldsymbol{F_\infty}$} if it is of type~$F_n$ for all~$n$, i.e.~if it is the fundamental group of an aspherical CW complex with finitely many cells in each dimension.

Geoghegan and Brown proved in \cite{BroGeo} that Thompson's group $F$ is of type~$F_\infty$.  Later, Brown provided some necessary and sufficient conditions for a given group acting on a complex $K$ to be of type~$F_n$, and used these criterion to prove that several groups are of type~$F_\infty$, including Thompson's groups $T$ and~$V$ \cite{Bro}. Since then, it has become common to use the discrete Morse theory of Bestvina and Brady~\cite{BesBra} to verify Brown's conditions.

 In this section, we apply these techniques to the complex $K(\G)$ associated to a rearrangement group~$\G$.  This involves understanding the connectivity of the descending links of vertices in this complex with respect to a discrete Morse function.  We use this technique in Subsection~\ref{subsec:morse} to prove that various rearrangement groups are finitely generated. Proving further finiteness properties requires some new technology for analyzing the connectivity of flag complexes, which we develop in Subsection~\ref{subsec:connectivity}.  We apply this technology in Subsection~\ref{subsec:finfty} to prove the following theorem.

\begin{theorem} \label{thm:finfty}
Let $\mathcal{R}$ be a replacement system with finite branching whose replacement graph is connected, and let\/ $\Gamma(\mathcal{R})$ be the associated graph family.  Suppose that, for every $m \geq 1$, all but finitely many of the graphs in\/ $\Gamma(\mathcal{R})$ have at least $m$ different collapsible subgraphs.  Then the corresponding rearrangement group is of type~$F_\infty$.
\end{theorem}

Here, $\mathcal{R}$ has \newword{finite branching} if there exists an upper bound on the degrees of vertices in the full expansion sequence for $\mathcal{R}$.  For example, if the initial and terminal vertices of~$R$ both have degree one, then $\R$ has finite branching.

As an application of this theorem, we prove in Subsection~\ref{subsec:finfty} that all of the rearrangement groups for fractals in the Vicsek family are of type~$F_\infty$.

\subsection{Brown's Criterion and Bestvina-Brady Morse Theory}
\label{subsec:morse}

If $K$ is any cubical complex, a \newword{Morse function} on $K$ is a map $\mu \colon K\to \mathbb{R}$ satisfying the following conditions:
\begin{enumerate}
\item The image under $\mu$ of the vertex set of $K$ is discrete.
\smallskip
\item No two adjacent vertices have the same value under~$\mu$.\smallskip
\item The map $\mu$ restricts to an affine linear function on each cube of~$K$.
\end{enumerate}
Given a Morse function $\mu$ on $K$ and a real number $t$, the corresponding \newword{sublevel complex} $K^{\leq t}$ is the subcomplex of $K$ consisting of all cubes that are entirely contained in $\mu^{-1}\bigl((-\infty,t]\bigr)$.  The \newword{descending link} $\mathrm{lk}_\downarrow (v,K)$ of a vertex $v$ in $K$ is its link in the appropriate sublevel complex, i.e.
\[
\mathrm{lk}_\downarrow (v,K) \;=\; \mathrm{lk}\bigl(v,K^{\leq \mu(v)}\bigr).
\]
The following theorem is a combination of the Bestvina-Brady Morse lemma~\cite{BesBra} with Brown's criterion for finiteness properties~\cite{Bro}.

\begin{theorem}[Bestvina-Brady-Brown]\label{thm:MorseTheory}  Let\/ $\G$ be a group acting properly by isometries on a contractible cube complex~$K$, let $\mu$ be a\/ $\G$-invariant Morse function on~$K$, and let $n\geq 1$.  Suppose that
\begin{enumerate}
\item Each sublevel complex $K^{\leq t}$ has finitely many orbits of cubes, and\smallskip
\item There exists a\/ $t \in \mathbb{R}$ so that the descending link of each vertex in $\mu^{-1}\bigl([t,\infty)\bigr)$ is $(n-1)$-connected.
\end{enumerate}
Then\/ $\G$ is of type $F_n$.\quad\qedsymbol
\end{theorem}

We wish to apply the above theorem to prove finiteness properties for rearrangement groups.  We begin by describing the flag complexes that will arise as descending links in the associated complex.

\begin{definition}Let $\R$ be a replacement system, and let $G$ be a graph in the associated graph family.  The \newword{contraction complex of $\boldsymbol{G}$ with respect to~$\boldsymbol{\R}$}, denoted $\mathrm{Con}(G,\R)$, is the flag complex defined as follows:
\begin{enumerate}
\item There is one vertex in $\mathrm{Con}(G,\R)$ for each characteristic map of $R$ into~$G$.
\item Two vertices in $\mathrm{Con}(G,\R)$ are connected by an edge if the corresponding characteristic maps do not \textbf{overlap}, i.e.~if the images of the two characteristic maps are edge disjoint.
\end{enumerate}
\end{definition}

For example, if $\R$ is the basilica replacement system and $G$ is the graph shown in Figure~\ref{fig:BasilicaCollapsibleSubgraphs}, then $\mathrm{Con}(G,\R)$ is a path of length two. If $\R$ is the Vicsek replacement system and $G$ is the graph shown in Figure~\ref{fig:VicsekCollapsibleSubgraphs}, then $\mathrm{Con}(G,\R)$ has sixteen vertices corresponding to the sixteen possible characteristic maps.  Each of the six vertices corresponding to the leftmost collapsible subgraph is connected by edges to each of the six vertices corresponding to the rightmost collapsible subgraph, so $\mathrm{Con}(G,\R)$ is the complete bipartite graph $K_{6,6}$ together with four isolated vertices.

In the context of rearrangement groups and the complex defined in Section~\ref{sec:complexes}, Theorem~\ref{thm:MorseTheory} takes the following form.

\begin{theorem}\label{thm:FinitenessProperties}Let\/ $\R$ be a replacement system, and let $n\geq 1$.  Suppose that\/ $\mathrm{Con}(G,\R)$ is $(n-1)$-connected for all but finitely many isomorphism types of graphs~$G\in\Gamma(\R)$.  Then the corresponding rearrangement group is of type~$F_n$.  \end{theorem}
\begin{proof}Let $\G$ be the rearrangement group associated to~$\R$, and let $K$ be the cubical complex $K(\G)$ defined in Section~\ref{sec:complexes}.  Define a Morse function $\mu\colon K\to\mathbb{R}$ by letting $\mu(v)$ be the rank of $v$ for each vertex~$v$, and then extending linearly to the cubes. Since the action of the rearrangement group $\G$ on the vertices of~$K$ preserves rank, this Morse function is $\G$-invariant.

We claim that each sublevel complex $K^{\leq t}$ has finitely many orbits of cubes.  Let $f_1 \colon X(G_0) \to X(G_1)$ and $f_2 \colon X(G_0) \to X(G_2)$ be rearrangements representing two vertices of the same rank.  If $G_1$ and $G_2$ are isomorphic, then for any base isomorphism $\varphi\colon G_1\to G_2$ the element $f_2^{-1} \circ \varphi \circ f_1$ of $\G$ takes $[f_2]$ to~$[f_1]$.  But there are finitely many isomorphism classes of graphs in $\Gamma(\R)$ with $t$ or fewer edges, so there are only finitely many orbits of vertices in~$K^{\leq t}$.  Since $K^{\leq t}$ is locally finite, it follows that $K^{\leq t}$ has only finitely many orbits of cubes.

Now let $t\in\mathbb{R}$ so that the contraction complex of any graph in $\Gamma(\R)$ with at least $t$ edges is $(n-1)$-connected.  If $[f]$ is any vertex in $\mu^{-1}\bigl([t,\infty)\bigr)$, then $f\colon X(G_0) \to X(G)$ for some graph $G$ with at least $t$ edges.  It follows from Theorem~\ref{thm:Links} that the descending link of $[f]$ in $K$ is isomorphic to $\mathrm{Con}(G,\R)$, and hence $\mathrm{link}_\downarrow([f],K)$ is $(n-1)$-connected.  By Theorem~\ref{thm:MorseTheory}, we conclude that $\G$ has type~$F_n$.
\end{proof}

\begin{example}[Finiteness Properties of the Basilica Rearrangement Group]
\label{ex:BasilicanotFinfty}
Using Theorem~\ref{thm:FinitenessProperties}, it is easy to show that the basilica Thompson group $T_B$ is finitely generated.  In particular, there are only three isomorphism types of graphs in the graph family for the basilica whose corresponding contraction complexes are disconnected.  These graphs are shown in Figure~\ref{fig:ThreeDisconnectedBasilica}.  A specific four-element generating set for $T_B$ is given in~\cite{BeFo}.
\begin{figure}[b]
\vgraphics{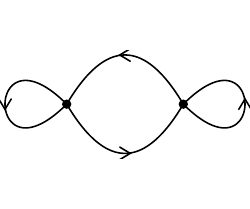}\;\;
\hfill
\vgraphics{BasilicaDifferentBase}
\hfill
\vgraphics{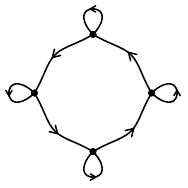}
\caption{The three graphs in the basilica family whose corresponding contraction complexes are disconnected.}
\label{fig:ThreeDisconnectedBasilica}
\end{figure}
\begin{figure}
\centering
\includegraphics{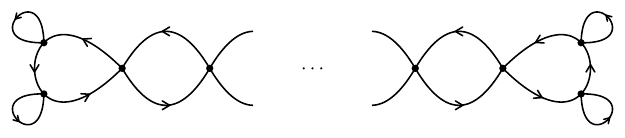}
\caption{A graph in the basilica graph family whose descending link is not simply connected.}
\label{fig:BasilicaNotFP}
\end{figure}

Theorem~\ref{thm:FinitenessProperties} cannot be used to show that $T_B$ is finitely presented.  In particular, consider the family of graphs shown in Figure~\ref{fig:BasilicaNotFP}, where there may be any number of intermediate two-cycles.  Each such graph has four vertices in its contraction complex, corresponding to the two collapsible subgraphs on the left and the two collapsible subgraphs on the right.  Two of these vertices are connected by an edge if and only if they are on different sides, so the contraction complex is a four-cycle, and is therefore not simply connected.  Witzel and Zaremsky have recently shown that $T_B$ is in fact not finitely presented~\cite{WiZa}.

Similar arguments can be made for all of the rearrangement groups in the rabbit family (see Example~\ref{ex:rabbit}).  That is, all of the rearrangement groups in this family are finitely generated, with only finitely many disconnected contraction complexes, but Theorem~\ref{thm:FinitenessProperties} cannot be used to show that they are finitely presented.
\end{example}

\begin{figure}
\centering
\subfloat[\label{subfig:FsemiZ2Replacement}]{\vgraphics{EdgeReplacement1} \quad\longvarrow\quad \vgraphics{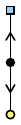}}
\qquad\qquad\qquad
\subfloat[\label{subfig:BadLinks}]{\vgraphics{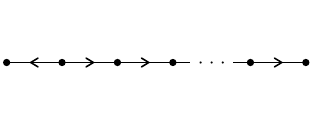}}
\caption{(a) A replacement rule for the replacement system whose rearrangement group is $F\rtimes\mathbb{Z}_2$.  (b) An infinite family of graphs whose contraction complexes are disconnected.}
\end{figure}
\begin{remark}The converse of Theorem~\ref{thm:FinitenessProperties} does not hold.  For example, consider the replacement system $\R$ having a single edge as the base graph, with the replacement rule shown in Figure~\subref*{subfig:FsemiZ2Replacement}.  The rearrangement group for~$\R$ is a semidirect product of Thompson's group $F$ with a cyclic group of order two, which is of type~$F_\infty$.  However, Figure~\subref*{subfig:BadLinks} shows an infinite family of graphs in the corresponding graph family whose corresponding contraction complexes consist of two disconnected vertices.
\end{remark}

%
%
\begin{remark}
Many rearrangement groups are non-finitely generated.  For example, consider the basilica replacement rule with a single edge as a base graph.  The rearrangement group is isomorphic to the union of the sequence
\[
F \;\;\subset\;\; F \wr_D F \;\;\subset\;\; F \wr_D F \wr_D F \;\;\subset\;\; \cdots
\]
of restricted wreath products, where $F$ is Thompson's group and $D$ is the set of all dyadic points in the interval~$(0,1)$. This group is not finitely generated, being an ascending union of proper subgroups.
\end{remark}

\subsection{Connectivity of Flag Complexes}
\label{subsec:connectivity}

In order to efficiently apply Theorem \ref{thm:MorseTheory} to rearrangement groups, we introduce two new pieces of technology for assessing the connectivity of flag complexes.

\begin{definition}Let $X$ be a simplical complex, and let $k\geq 1$.
\begin{enumerate}
\item  A simplex $\Delta$ in $X$ is called a \newword{$\boldsymbol{k}$-ground} for $X$ if every vertex of $X$ is adjacent to all but at most $k$ vertices in~$\Delta$.
\item We say that $X$ is \newword{$\boldsymbol{(n,k)}$-grounded} if there exists an $n$-simplex in $X$ that is a $k$-ground for~$X$.
\end{enumerate}
\end{definition}

Note that any face of a $k$-ground for $X$ is again a $k$-ground for~$X$.  Thus, an $(n,k)$-grounded complex is also $(n',k)$-grounded for all $n'<n$.

\begin{theorem}\label{thm:GroundedConnectivity}For $m,k\geq 1$, every finite\/ $(mk,k)$-grounded flag complex is\/ $(m-1)$-connected.
\end{theorem}
\begin{proof}We proceed by induction on $m$.  For $m=1$, the statement is that every finite $(k,k)$-grounded flag complex is connected, which is clear from the definition.

Now suppose that every finite $(mk,k)$-grounded flag complex is $(m-1)$-connected, and let $X$ be a finite $\bigl((m+1)k,k\bigr)$-grounded flag complex.  Then we can filter $X$ by a chain of flag complexes
\[
\Delta = X_0 \subset X_1 \subset \cdots \subset X_p = X.
\]
where $\Delta$ is a simplex of dimension $(m+1)k$ that is a $k$-ground for~$X$, and each~$X_i$ is obtained from $X_{i-1}$ by adding a single vertex~$v_i$.

Let $L_i$ denote the link of $v_i$ in $X_i$, and observe that each $X_i$ is homeomorphic to the union $X_{i-1} \cup_{L_i} CL_i$, where $CL_i$ denotes the cone on~$L_i$.  Since $\Delta$ is a \mbox{$k$-ground} for~$X$, we know that $L_i$ includes at least $mk+1$ vertices of~$\Delta$.  In particular, the intersection $L_i \cap \Delta$ contains an \mbox{$mk$-simplex}, which must be a $k$-ground for~$L_i$. By our induction hypothesis, it follows that each $L_i$ is \mbox{$(m-1)$-connected}.  Since $X_0 = \Delta$ is contractible, this proves that $X_i$ is \mbox{$m$-connected} for every~$i$, and in particular $X$ is \mbox{$m$-connected}.
\end{proof}

The criterion given in Theorem~\ref{thm:GroundedConnectivity} has already proven itself useful outside of the present context.  In~\cite{BeMa}, the first author and F.~Matucci use this criterion to prove that R\"{o}ver's simple group has type~$F_\infty$.  Also, M.~Zaremsky uses a generalization of the criterion developed here in~\cite{Za} to compute the $\Sigma$-invariants of generalized Thompson's groups.

Instead of applying Theorem~\ref{thm:GroundedConnectivity} directly to rearrangement groups, we will use it to derive a simpler criterion that meets our needs.

\begin{definition}A flag complex $X$ is \newword{uniformly $\boldsymbol{k}$-dense} if every vertex of~$X$ is adjacent to all but at most $k$ other vertices.
\end{definition}

\begin{theorem}\label{thm:DenseConnectivity}Every uniformly $k$-dense flag complex with at least $mk(k+1)+1$ vertices is $(m-1)$-connected.
\end{theorem}

Theorem~\ref{thm:DenseConnectivity} follows almost immediately from the following lemma.

\begin{lemma}\label{lem:DenseGrounded} Every uniformly $k$-dense flag complex with at least $n(k+1)+1$ vertices is $(n,k)$-grounded.
\end{lemma}

\begin{proof}  Let $X$ be a uniformly $k$-dense flag complex with at least $n(k+1)+1$ vertices. Then every simplex in $X$ is a $k$-ground, so it suffices to show that $X$ has at least one $n$-simplex.

We choose the vertices $v_0,v_1,\ldots v_n$ of this simplex inductively.  First, let $v_0$ be any vertex of $X$.  Now suppose that we have already chosen vertices $v_0,\ldots,v_{j-1}$ for some $j\leq n$.  Each $v_i$ is adjacent to all but at most $k+1$ vertices of $X$, including $v_i$ itself.  Then all but at most $j(k+1)$ vertices of $X$ are adjacent to all of the vertices $v_0,\ldots,v_{j-1}$.  Since
\[
\bigl(n(k+1)+1\bigr) - j(k+1) \;=\; (n-j)(k+1)+1 \;\geq\; 1,
\]
there is at least one vertex $v_j$ of $X$ that is adjacent to all of the vertices $v_0,\ldots,v_{k-1}$.  Continuing in this fashion, we arrive at an $n$-simplex $\{v_0,\ldots,v_n\}$ in~$X$.
\end{proof}

\begin{proof}[Proof of Theorem \ref{thm:DenseConnectivity}] Let $m\geq 0$, and let $X$ be a uniformly $k$-dense flag complex with at least $mk(k+1)+1$ vertices.  Let $n=mk$.  Then $X$ has at least $n(k+1)+1$ vertices, so by Lemma~\ref{lem:DenseGrounded}, $X$ is $(n,k)$-grounded.  Then $X$ is $(m-1)$-connected by Theorem~\ref{thm:GroundedConnectivity}.
\end{proof}

\subsection{Rearrangement Groups of Type $F_\infty$}
\label{subsec:finfty}

In this subsection we prove Theorem~\ref{thm:finfty} and use it to show that all of the rearrangement groups in a certain infinite family have type~$F_\infty$.

Finite branching is particularly fundamental to our arguments. If $\mathcal{R}$ has finite branching, then there is a uniform upper bound on the degrees of vertices for graphs $G \in \Gamma(\mathcal{R})$, for any such $G$ has an expansion that is isomorphic to a graph in the full expansion sequence.  This allows us to bound the number of collapsible subgraphs that can intersect in any such~$G$.

For the following proposition, we say that two collapsible subgraphs of a graph \newword{overlap} if they have an edge in common.

\begin{lemma}
Let\/ $\mathcal{R}$ be a replacement system with finite branching whose replacement graph $R$ is connected.  Then there exists an\/ $i\in\mathbb{N}$ such that, for every graph $G\in\Gamma(\R)$, each collapsible subgraph of $G$ overlaps with at most\/ $i$ other collapsible subgraphs.
\label{lem:bound}
\end{lemma}

\begin{proof}
Let $d$ be the diameter of the graph $R$ and let $k$ be the upper bound on the degrees of vertices of graphs in~$\Gamma(\R)$.  Let $G\in\Gamma(\R)$, and let $S$ be a collapsible subgraph of~$G$.  Since $R$ is connected, any collapsible subgraph of $G$ that overlaps with $S$ must be contained in a ball of radius $2d$ centered at any vertex in~$S$.  Such a ball has at most $k^{2d}$ edges, so $S$ intersects at most $2^{k^{2d}}$ other collapsible subgraphs.
\end{proof}

We are now in a position to prove our main theorem for this section.

\begin{proof}[Proof of Theorem \ref{thm:finfty}]
Let $\R = (G_0,e\to R)$ be a replacement system with finite branching whose replacement graph $R$ is connected, and suppose that for every $m\geq 1$, all but finitely many of the graphs of $\Gamma(R)$ have at least $m$ collapsible subgraphs.  We must show that the corresponding rearrangement group has type~$F_\infty$.

Observe first that any collapsible subgraph of any graph in $\Gamma(\R)$ corresponds to at most $j$ different characteristic maps, where $j$ is the maximum of the orders of the automorphism groups of $R$ and~$\Rloop$.  If $i$ is the upper bound provided by Lemma~\ref{lem:bound}, it follows that each characteristic map of $R$ into any graph in $\Gamma(\R)$ overlaps with at most $k = (i+1)j - 1$ other characteristic maps.  Thus the contraction complex of any graph in $\Gamma(\R)$ is uniformly $k$-dense.

Let $n\geq 1$.  By Theorem~\ref{thm:DenseConnectivity}, the complex $\mathrm{Con}(G,\R)$ is $(n-1)$-connected for every graph $G\in\Gamma(\R)$ with at least $nk(k+1)+1$ collapsible subgraphs.  Therefore, by Theorem~\ref{thm:FinitenessProperties}, the corresponding rearrangement group has type~$F_n$.  This holds for all~$n$, so the corresponding rearrangement group has type~$F_{\infty}$.
\end{proof}

We now apply Theorem~\ref{thm:finfty} to the Vicsek family rearrangement groups described in Example~\ref{ex:VicsekFamily}.

\begin{theorem} \label{thm:vicsekfinfty}
The rearrangement groups for the Vicsek family are all of type~$F_{\infty}$.
\end{theorem}
\begin{proof}Let $\R_n$ denote the replacement system for the $n$th Vicsek rearrangement group, as shown in Table~\ref{tab:VicsekFamily}. Note that $\R_n$ has finite branching and that the replacement graph is connected. Thus, by Theorem~\ref{thm:finfty}, it suffices to show that for each~$m \geq 1$, all but finitely many graphs in the graph family $\Gamma(\R_n)$ have at least $m$ collapsible subgraphs.

\begin{figure}
\centering
\subfloat[]{\vgraphics{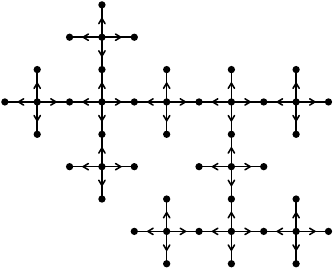}}
\hfill
\subfloat[]{\vgraphics{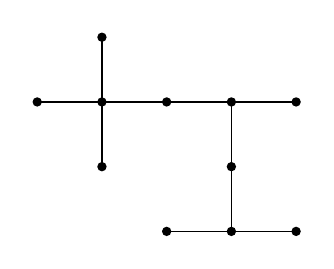}}
\caption{(a) A graph $G$ in the Vicsek graph family. (b) The corresponding tree $T_G$.}
\label{fig:VicsekStructureTree}
\end{figure}
The graphs in the graph family $\Gamma(\mathcal{R}_n)$ were described in Example~\ref{ex:VicsekGraphFamily}.  Given any $G\in\Gamma(\mathcal{R}_n)$, let $T_G$ be the tree that has one vertex for each source in $G$ and an edge between two vertices if the corresponding sources are a distance two apart in $G$.  For example, Figure~\ref{fig:VicsekStructureTree} shows a graph $G$ in $\Gamma(\R_4)$ and the corresponding tree~$T_G$.

Now, observe that there is a collapsible subgraph in $G$ for each vertex of degree one in $T_G$ and two collapsible subgraphs in $G$ for each vertex of degree two in $T_G$.  But at least half the vertices in any finite tree have degree 1 or~2, so $G$ will have at least $m$ collapsible subgraphs as long as $T_G$ has at least $2m$ vertices.  This occurs whenever $G$ has at least $2mn$ edges, and therefore $G$ has at least $m$ collapsible subgraphs for all but finitely many~$G$.
\end{proof}

\subsection*{Acknowledgements}
The authors would like to thank Collin Bleak, Kai-Uwe Bux, Marco Marschler, Francesco Matucci, Stefan Witzel, and Matt Zaremsky for many helpful conversations and suggestions.  We would also like to thank the anonymous referee for several helpful comments.

\bigskip
\newcommand{\arxiv}[1]{\href{https://arxiv.org/abs/#1}{arXiv:#1}}
\newcommand{\doi}[1]{\href{https://doi.org/#1}{doi:#1}}
\bibliographystyle{plain}

\begin{thebibliography}{10}

\smallskip

\bibitem{BeFo}
J.~Belk and B.~Forrest, A Thompson Group for the Basilica. \textit{Groups, Geometry, and Dynamics} \textbf{9.4} (2015): 975--1000. \doi{10.4171/GGD/333}.

\smallskip
\bibitem{BeMa}
J.~Belk and F.~Matucci, R\"{o}ver's Simple Group is of Type~$F_\infty$.  \textit{Publicacions Matem\'{a}tiques} \textbf{60.2} (2016): 501--524. \doi{10.5565/PUBLMAT\_60216\_07}.

\smallskip

\bibitem{BesBra}
M.~Bestvina and N.~Brady, Morse Theory and Finiteness Properties of Groups. \textit{Inventiones Mathematicae} \textbf{129}, no.~3 (1997): 445--470. \doi{10.1007/s002220050168}.

\smallskip

\bibitem{Bou}
N.~Bourbaki, \textit{General Topology, Chapters 5--10}. Springer-Verlag Berlin Heidelber, 1998.

\smallskip

\bibitem{BriHae}
M.~Bridson and A.~Haefliger, \textit{Metric Spaces of Non-Positive Curvature}. Grundlehren der Mathematischen Wissenschaften \textbf{319}. Springer-Verlag Berlin Heidelberg, 1999. \doi{10.1007/978-3-662-12494-9}.

\smallskip

\bibitem{Bro}
K.~Brown, Finiteness Properties of Groups.  \textit{Journal of Pure and Applied Algebra} \textbf{44.1} (1987): 45--75. \doi{10.1016/0022-4049(87)90015-6}.

\smallskip

\bibitem{BroGeo}
K.~Brown and R.~Geoghegan, An Infinite-Dimensional Torsion-Free $\mathrm{FP}_\infty$ Group. \textit{Inventiones Mathematicae} \textbf{77}, no.~2 (1984): 367--381. \doi{10.1007/BF01388451}.



\smallskip

\bibitem{CFP}
J.~Cannon, W.~Floyd, and W.~Parry, Introductory Notes on Richard Thompson's Groups. \textit{Enseignement Math\'{e}matique} \textbf{42} (1996): 215--256.

\smallskip

\bibitem{Ch}
J.~Charatonik, Monotone Mappings of Universal Dendrites. \textit{Topology and its Applications} \textbf{38}, no.~2 (1991): 163--187. \doi{10.1016/0166-8641(91)90083-X}

\smallskip

\bibitem{DH}
A.~Douady and J.~Hubbard, \'{E}tude Dynamique des Polyn\^{o}mes Complexes, Part I. \textit{Publ Math.\ Orsay} 1984--1985.

\smallskip


\bibitem{Farley1}
D.~Farley, Finiteness and $\CAT(0)$ Properties of Diagram Groups. \textit{Topology} \textbf{42.5} (2003): 1065--1082. \doi{10.1016/S0040-9383(02)00029-0}

\smallskip

\bibitem{Farley2}
D.~Farley, Actions of Picture Groups on $\CAT(0)$ Cubical Complexes. \textit{Geometriae Dedicata} \textbf{110.1} (2005): 221--242. \doi{10.1007/s10711-004-1530-z}.

\smallskip

\bibitem{Farley3}
D.~Farley, Homological and Finiteness Properties of Picture Groups. \textit{Transactions of the American Mathematical Society} \textbf{357.9} (2005): 3567--3584. \doi{10.1090/S0002-9947-04-03720-1}


\smallskip

\bibitem{GuSa}
V.~Guba and M.~Sapir, \textit{Diagram Groups}. Memoirs of the American Mathematical Society \textbf{130}, no.~620.  American Mathematical Society,~1997. \doi{10.1090/memo/0620}.

\smallskip



\bibitem{Ki}
J.~Kigami, \textit{Analysis on Fractals}. Cambridge Tracts in Mathematics \textbf{143}. Cambridge University Press,~2001. \doi{10.1017/CBO9780511470943}

\smallskip
\bibitem{MiLe}
J.~Milnor, Geometry and Dynamics of Quadratic Rational Maps, with an Appendix by Milnor and Tan Lei. \textit{Experimental Mathematics}~\textbf{2}, no.~1 (1993): 37--83. \doi{10.1080/10586458.1993.10504267}

\smallskip

\bibitem{Ne}
V.~Nekrashevych, \textit{Self-Similar Groups}. Mathematical Surveys and Monographs \textbf{117}. American Mathematical Society,~2005. \doi{10.1090/surv/117}

\smallskip

\bibitem{St}
R.~Strichartz, Analysis on Fractals. \textit{Notices of the American Mathematical Society} \textbf{46}, no.~10 (1999): 1199-1208. \\ \href{http://www.ams.org/notices/199910/fea-strichartz.pdf}{http://www.ams.org/notices/199910/fea-strichartz.pdf}

\smallskip

\bibitem{Th}
W.~P.~Thurston, On the geometry and dynamics of iterated rational maps. In \textit{Complex Dynamics} (D.~Schleicher, N.~Selinger, eds.), A~K~Peters, Wellesley, MA 2009, 3--137.

\smallskip

\bibitem{WiZa}
S.~Witzel, M.~Zaremsky, The Basilica Thompson Group is not Finitely Presented. Preprint (2016). \arxiv{1603.01150}

\smallskip

\bibitem{Za}
M. Zaremsky, On the $\Sigma$-invariants of Generalized Thompson Groups and Houghton Groups.  Preprint (2015).  \arxiv{1502.02620}

\end{thebibliography}

\end{document}